\documentclass[a4paper]{amsart}
\usepackage{ a4wide} 
\usepackage{amssymb}
\usepackage{amsmath, amsthm, amsfonts}
\usepackage{latexsym}
\usepackage{pdfsync}
\usepackage{pdfsync}
\usepackage{xypic}
\usepackage[matrix, arrow]{xy}
\usepackage{array}

\usepackage{fancyhdr}
\usepackage{tabularx}

\theoremstyle{plain}
\newtheorem{theorem}{Theorem}[section]
\newtheorem{corollary}[theorem]{Corollary}
\newtheorem{lemma}[theorem]{Lemma}
\newtheorem*{lemma*}{Lemma}
\newtheorem{proposition}[theorem]{Proposition}
\newtheorem*{proposition*}{Propositon}

\theoremstyle{definition}
\newtheorem{definition}[theorem]{Definition}

\newtheorem{remark}[theorem]{Remark}
\newtheorem*{remark*}{Remark}

\newtheorem{remark/notation}[theorem]{Remark/Notation}
\newtheorem{remark/definition}[theorem]{Remark/Definition}
\newtheorem{remark/example}[theorem]{Remark/Example}
\newtheorem{example}[theorem]{Example}

\newtheorem*{example*}{Example}

\newtheorem*{claim*}{Claim}

\newcommand{\id}{\operatorname{Id}}
\newcommand{\ind}{\operatorname{ind}}
\newcommand{\supp}{\operatorname{supp}}
\newcommand{\res}{\operatorname{res}}

\newcommand{\GL}{\operatorname{GL}}
\newcommand{\Hom}{\operatorname{Hom}}
\newcommand{\Prim}{\operatorname{Prim}}
\newcommand{\Rep}{\operatorname{Rep}}
\newcommand{\Ad}{\operatorname{Ad}}
\newcommand{\ad}{\operatorname{ad}}

\newcommand{\Sp}{\operatorname{Sp}}
\newcommand{\R}{\operatorname{\mathbb{R}}}
\newcommand{\Q}{\operatorname{\mathbb{Q}}}

\newcommand{\Z}{\operatorname{\mathbb{Z}}}
\newcommand{\N}{\operatorname{\mathbb{N}}}

\newcommand{\RR}{\operatorname{\mathbb{R}}}
\newcommand{\QQ}{\operatorname{\mathbb{Q}}}
\newcommand{\CC}{\operatorname{\mathbb{C}}}
\newcommand{\ZZ}{\operatorname{\mathbb{Z}}}
\newcommand{\NN}{\operatorname{\mathbb{N}}}
\newcommand{\TT}{\operatorname{\mathbb{T}}}

\newcommand{\G}{\operatorname{\mathfrak{g}}}
\def\Char{\operatorname{char}}
\def\P{\mathbb P}

\newcommand{\m}{\mathfrak{m}}
\newcommand{\n}{\mathfrak{n}}
\newcommand{\z}{\mathfrak{z}}
\newcommand{\h}{\mathfrak{h}}
\def\a{\mathfrak{a}}
\def\j{\mathfrak{j}}
\newcommand{\rk}{\rangle}
\newcommand{\lk}{\langle}
\newcommand{\B}{\mathcal B}
\newcommand{\U}{\mathcal U}

\def\r{\mathfrak{r}}
\def\s{\mathfrak{s}}
\def\u{\mathfrak{u}}
\def\k{\mathfrak{k}}

\def\K{\mathcal K}
\begin{document}

\title[Kirillov Theory]
  {A general Kirillov Theory for locally compact nilpotent groups}

\author[Echterhoff]{Siegfried
  Echterhoff}
   \address{Westf\"alische Wilhelms-Universit\"at M\"unster,
  Mathematisches Institut, Einsteinstr. 62 D-48149 M\"unster, Germany}
\email{echters@uni-muenster.de}

  \author[Kl\"uver]{Helma
  Kl\"uver}
  \address{Westf\"alische Wilhelms-Universit\"at M\"unster,
  Mathematisches Institut, Einsteinstr. 62 D-48149 M\"unster, Germany}
\email{Helma$\underline{\ }$Kluever@gmx.de}

\begin{abstract} We discuss a very general Kirillov Theory for the representations of
certain nilpotent groups which gives a combined view an many known examples from the literature.
\end{abstract}

  \thanks{The research for this paper was partially supported by the German Research Foundation
  (SFB 478 and  SFB 878) and the EU-Network Quantum SpacesÐNoncommutative Geometry (Contract
No. HPRN-CT-2002-00280).}
\maketitle

\section{Introduction}
In this paper we want to present a quite general version of Kirillov's orbit method for the representations of second countable locally compact nilpotent groups, which resulted out of an attempt to understand the results of Howe in \cite{How}.
In particular, we wanted to understand Howe's version of a Kirillov theory for unipotent groups 
over function fields with ``small'' nilpotence length, which might be useful in the 
 study of the Baum-Connes conjecture for linear algebraic groups over such fields, following the ideas of \cite{CEN}.
Although we have to admit that we still struggle with some details in  Howe's paper, we learned 
enough from his ideas to find a way to formulate  a quite general version of 
Kirillov's theory which covers a big class of nilpotent locally compact groups, containing 
\begin{enumerate}
\item connected and simply connected real nilpotent Lie groups (the classical situation studied by 
 Kirillov \cite{Kir});
\item unipotent groups over $\QQ_p$ (which have been studied by Moore in \cite{Moo} and by Boyarchenko and Sabitova in \cite{BS});
\item Quasi-$p$ groups with ``small'' nilpotence length 
(studied by Howe in \cite{How}; but see also \cite{BS});
\item countable torsion free divisible groups (which have been studied by Carey, Moran and Pearce in \cite{CMP}).
\end{enumerate}
Following the ideas of Howe in \cite{How} we 
will generalize the notion of a Lie group by generalizing the classical notion of a Lie algebra.
For this we introduce the notion of nilpotent $k$-Lie pairs $(G,\G)$ 
for some $k\in \NN\cup\{\infty\}$ in which
$G$ is a locally compact nilpotent group of nilpotence length $l\leq k$ and $\G$ is a Lie-algebra
over the ring $\Lambda_k:=\ZZ[\frac{1}{k!}]$ if $k\in \NN$ and $\Lambda_k:=\QQ$ if $k=\infty$
such that $G$ and $\G$ can be identified via a
bijective homeomorphism $\exp:G\to \G$ which  satisfies the Campbell-Hausdorff formula.
In order to make things work, we need some other technical 
ingredients, which are explained 
in detail in \S \ref{section-Lie-pair-construction} below.  A very important one is the existence of a
locally compact $\Lambda_k$-module $\mathfrak{m}$ together with a basic character $\epsilon:\mathfrak{m}\to \TT$ such that the dual $\G^*:=\Hom_{\Lambda_k}(\G,\mathfrak{m})$ is isomorphic to 
the Pontrjiagin-dual $\widehat{\G}$ of the (locally compact) Lie algebra $\G$ via the map
$$\G^*\to \widehat{\G};f\mapsto \epsilon\circ f.$$
We then say that $(G,\G)$ is {\em $(\m,\epsilon)$-dualizable}.
For instance, if $G$ is a classical connected and simply connected 
nilpotent Lie-group with Lie-algebra $\G$, then 
$(G,\G)$ becomes an $(\RR,\epsilon)$-dualizable  nilpotent $k$-Lie pair for any $k\geq 2$ 
by taking $\epsilon:\RR\to \TT$ to be the basic character $\epsilon(t)=e^{2\pi it}$ (but any other character
of $\RR$ will do the job, too).

Having these data, then for each $f\in \G^*$ we can define in a more or 
less straightforward way the concept of {\em standard polarizing 
subalgebras} $\mathfrak{r}$ of $\G$ such that $f$ determines a character $\varphi_f\in \widehat{R}$
for $R=\exp(\mathfrak{r})$, and, using  the ideas of Kirillov and Howe, we can show that
the induced representations $\ind_R^G\varphi_f$ are always irreducible.
So far, everything works almost as in the classical Lie group case. But in our general setting, we 
cannot expect that every irreducible representation is induced, and it is also not clear that the induced representation 
$\ind_R^G\varphi_f$ is independent 
of the choice of the standard polarization $\mathfrak{r}$. However, it turns out that the primitive ideal
$P_f:=\ker (\ind_{R}^G\varphi_f)$ in the primitive ideal space $\Prim(C^*(G))$ of the group C*-algebra
$C^*(G)$ of $G$  is independent of the choice of $\mathfrak{r}$. We therefore get
a well-defined {\em Kirillov map}
$$\kappa: \G^*\to \Prim(C^*(G)); f\mapsto P_f=\ker (\ind_{R}^G\varphi_f).$$
There is a canonical coadjoint action $\Ad^*:G\to\GL(\G)$ and it is easily checked that the 
Kirillov map $\kappa$ is constant on $\Ad^*(G)$-orbits in $\G^*$. Moreover, we shall see that 
$\kappa$ is always continuous with respect to the given topology on $\G^*\cong\widehat{\G}$,
so that the Kirillov map will actually be constant on $\Ad^*(G)$-quasi orbits in $\G^*$ (two elements 
$f,f'\in \G^*$ lie in the same $\Ad^*(G)$-quasi orbit if and only if $f\in \overline{\Ad^*(G)f'}$ and 
$f'\in \overline{\Ad^*(G)f}$ --- we then write  $f\sim f'$). 
We therefore obtain a well defined continuous  map
$$\tilde{\kappa}: \G^*\!\!/\!\!\sim\;\to \Prim(C^*(G)); [f]\mapsto P_f,$$
which we call the {\em Kirillov-orbit map}. 
The main results of this paper are the following:

\begin{theorem}\label{thm-surjective}
Let $(G,\G)$ be an $(\m,\epsilon)$-dualizable  nilpotent $k$-Lie-pair. Then the Kirillov-orbit map 
$\tilde{\kappa}: \G^*\!\!/\!\!\sim\;\to \Prim(C^*(G))$ is continuous and surjective.
\end{theorem}

Unfortunately, we can show injectivity of the Kirillov-orbit map only under the additional assumption, 
that the sum $\h+\n$ of two closed subalgebras $\h,\n\subseteq \G$ with $[\h,\n]\subseteq \n$ is again closed in $\G$. If this is satisfied, we say that $(G,\G)$ is {\em regular}.  It is easy to check that all examples mentioned in the above list are regular. For regular Lie pairs $(G,\G)$ we can use ideas of Joy \cite{Joy} to prove

\begin{theorem}\label{thm-bijective}
Suppose that $(G,\G)$ is a regular  $(\m,\epsilon)$-dualizable 
nilpotent $k$-Lie-pair. Then the Kirillov-orbit map
$\tilde{\kappa}: \G^*\!\!/\!\!\sim\;\to \Prim(C^*(G))$
is a homeomorphism.
\end{theorem}

Of course, having these results, it is interesting to know under which conditions the Kirillov orbit method 
not only computes the primitive ideals  but also the irreducible representations of  $G$. Recall that it follows from Glimm's famous theorem (e.g., see \cite[Chapter 12]{Dix}) that an irreducible $*$-representation $\pi:A\to \B(H_\pi)$ of a separable C*-algebra $A$ is completely determined by its kernel 
$\ker\pi\in \Prim(A)$ if and only if its image contains the compact operators $\K(H_\pi)$. We call such representations {\em GCR-representations}.
If, moreover, $\pi(A)$ is equal to $\K(H_\pi)$, we say that  $\pi$ is a {\em CCR-representation}.
Glimm's result also implies that $\pi$ is GCR (resp. CCR) if and only if the point-set $\{\pi\}$ is locally closed\footnote{Recall that a subset $Y$ of a topological space $X$ is called {\em locally closed} if $Y$ is open in its closure $\overline{Y}$.}
(resp. closed) in $\widehat{A}$. For representations of a second countable locally compact group $G$, these notations carry over by the canonical identification of the unitary representations of $G$ with the $*$-representations of the group C*-algebra $C^*(G)$. We are able to show:

\begin{theorem}\label{thm-GCR}
Suppose that $(G,\G)$ is a regular $(\m,\epsilon)$-dualizable 
$k$-Lie pair. Let $f\in \G^*$  such that $\Ad^*(G)f$ is locally closed (resp. closed) in $\G^*$. Then $\pi_f=\ind_R^G\varphi_f$ is GCR (resp. CCR), where $R=\exp(\r)$ for a standard $f$-polarizing subalgebra  $\r$ of $\G$. In particular, if all $\Ad^*(G)$-orbits are locally closed in $\G^*$, we obtain a well defined homeomorphism
$$\widehat{\kappa}:\G^*/\Ad^*(G)\to \widehat{G}; \; \Ad^*(G)f\mapsto\pi_f$$
\end{theorem}

The paper is organized as follows: after this introductory section we start in \S \ref{sec-prel} with some preliminaries on induced representations of groups and the Fell-topology on the spaces of all equivalence classes of unitary representations. After that, in \S \ref{sec-Howe} we proceed with some material on representations of general nilpotent groups and, in particular, two-step nilpotent groups, which provides the base for the following sections. Most of the material in that section is based on work of Howe in \cite{How}
and might be well-known to the experts, but for completeness and for the readers convenience  we decided to present complete proofs. In \S \ref{section-Lie-pair-construction} we introduce our notion of 
$(\m,\epsilon)$-dualizable $k$-Lie pairs, and we present some of the basic properties, before we 
construct the Kirillov map in \S \ref{section-Kirillov-map}. The main results on (bi)-continuity and bijectivity of the Kirillov-orbit map are presented in \S \ref{section-the-kirillov-homeomorphism}. 
In  \S \ref{sec-GCR} we give the proof of Theorem \ref{thm-GCR}. 
The examples listed above are discussed in detail in \S \ref{section-examples} before we close this paper with some technical details on the Campbell-Hausdorff formula which are needed in the body of the paper.

This paper is partly based on the doctoral thesis 
of the second named author which was written under the direction of the first named author.

\section{Some preliminaries on induced representations}\label{sec-prel}

Let $G$ be a locally compact group, $H$ a closed
subgroup of $G$, $q:G \rightarrow G/H$ the canonical quotient map,
and $\sigma$ a unitary representation of $H$ on the Hilbert space
$H_{\sigma}$. Let us briefly recall the definition of the induced representation 
$\ind_H^G\sigma$ of $G$, where we want to use Blattner's construction  as 
introduced in \cite{Blat}. As a general reference we refer to Folland's book \cite{Fol}.
Let $\gamma:H\to (0,\infty)$ be given by 
$
\gamma(h)=\left({\Delta_G(h)}/{\Delta_H(h)}\right)^{\frac{1}{2}}
$
where $\Delta_G$ and $\Delta_H$ denote the modular functions of $G$ and $H$, respectively.
Note that nilpotent locally compact groups are unimodular (see \cite[Corollary 2]{How}), so that the function 
$\gamma$ will be trivial in the body of this paper. Let
\begin{align*}
\mathcal{F}_{\sigma}:=\{\xi \in C(G,H_{\sigma}) \mid q(\supp(f))\subseteq G/H &\;\text{is compact, and}\\
& \xi(x h)= \gamma(h^{-1})\sigma(h^{-1})\xi(x) \; \hbox{for} \; x \in G, h \in H  \}.
\end{align*}
If $\beta: G\to [0,\infty)$ is a Bruhat-section for
$G/H$, we may define an inner product $\lk \xi,\eta\rk\in \CC$ for $\xi,\eta\in \mathcal F_\sigma$ by
\begin{equation}\label{definition-pairing-on-F-sigma}
\langle \xi,\eta\rangle := \int_{G} \beta (x) \langle \xi(x), \eta(x)
\rangle_\sigma \; d\mu(x).
\end{equation}
We denote by $H_{\ind \sigma}$ the Hilbert space completion of
$\mathcal{F}_{\sigma}$ with respect to this inner product. The induced unitary  representation
$\ind_H^G\sigma$ of $G$ on $H_{\ind\sigma}$ is then given by
$$\big(\ind_H^G\sigma(x)\xi\big)(y)=\xi(x^{-1}y).$$
\begin{remark}\label{rem-properties-of-induction}
The following basic properties of the induced representation can be found in Folland's book \cite{Fol}. We shall use them throughout without further reference:
\\
{\bf (a)} (Induction in steps) Suppose $L\subseteq H\subseteq G$ are closed subgroups of $G$,
then $\ind_L^G \sigma\cong \ind_H^G( \ind_L^H \sigma)$.
\\
{\bf (b)}
(Induction-restriction)  If $\pi$ is a unitary representation of $G$ and
$\sigma$ a unitary representation of $H\subseteq G$, then
$
\ind_H^G (\pi|_H \otimes \sigma) \cong \pi \otimes \ind_H^G \sigma
$. 
In particular, with  $\sigma=1_H$, we get $\ind_H^G(\pi|_H)\cong \ind_H^G1_H\otimes\pi$.
\\
{\bf (c)}
Let $L$ be a closed normal subgroup of 
$G$ and let $H$ be a closed subgroup of $G$ with $L \subseteq H$. 
Then \begin{equation*}
\ind_H^G (\sigma \circ q) \cong (\ind_{H/L}^{G/L} \sigma) \circ q,
\end{equation*}
for any unitary representation $\sigma$ of $H/L$,
where $q: G \rightarrow G/L$ denotes the quotient map.
\\
{\bf (d)} If $\pi$ is a representation of the closed subgroup $H$ of $G$ and $x\in G$, then 
$\ind_{xHx^{-1}}^G(x\cdot \pi)\cong\ind_H^G\pi$, where $x\cdot \pi$ denotes the unitary representation
of $xHx^{-1}$ defined by $x\cdot \pi(y)=\pi(x^{-1}yx)$.
\end{remark}

If $H$ is a closed subgroup of $G$, then $G$ acts on $C_0(G/H)$ by left translation. 
A covariant representation of the C*-dynamical system $(G, C_0(G/H))$ consists of a pair $(\pi, P)$,
where $\pi:G\to \U(H_\pi)$ is a unitary representation of $G$ and $P:C_0(G/H)\to \B(H_\pi)$ is a non-degenerate $*$-representation on the same Hilbert space such that
$$P(x\cdot \varphi)=\pi(x)P(\varphi)\pi(x^{-1})$$
for all $\varphi\in C_0(G/H)$ and $x\in G$, where $x\cdot \varphi(yH)=\varphi(x^{-1}yH)$.
By Schur's lemma, a covariant representation $(\pi,P)$ is irreducible, iff every intertwiner for the pair $(\pi,P)$ 
is a multiple of the identity.

If $\sigma$ is a unitary representation of  $H$, then $\sigma$ induces to a covariant representation
 $(\ind_H^G\sigma, P^\sigma)$, where $\ind_H^G\sigma$ is the induced representation as explained above, and  
$P^\sigma:C_0(G/H)\to\B(H_{\ind\sigma})$ is given by
$\big(P^\sigma(\varphi)\xi\big)(x)=\varphi(xH)\xi(x)$. 
Note that any intertwiner $T: H_\sigma\to H_\rho$ for two unitary representations $\sigma$ and $\rho$ of $H$
induces an intertwiner $\tilde{T}$ for the induced covariant pairs $(\ind_H^G\sigma,P^\sigma)$ and 
$(\ind_H^G\rho, P^\rho)$ via $(\tilde{T}\xi)(x)=T(\xi(x))$.
We shall make extensive use of

\begin{theorem}[Mackey's imprimitivity theorem] \label{thm-Mackey}
The assignment
$\sigma\mapsto (\ind_H^G\sigma, P^\sigma)$
induces a bijective correspondence between the collection $\Rep(H)$ of all equivalence classes of 
unitary representations of $H$ and the collection $\Rep(G, C_0(G/H))$ of all equivalence classes of 
covariant representations of $(G, C_0(G/H))$.
Moreover, the induced pair $(\ind_H^G\sigma, P^\sigma)$ is irreducible if and only if $\sigma$ is irreducible.
\end{theorem}

We shall later need the following  lemma. It might be well-known to experts, but by lack of a reference, we shall give the proof.

\begin{lemma}\label{lem-ind-res}
Suppose that $G$ is a locally compact group, $R$ and $N$ are closed subgroups of $G$ such that 
$N$ is normal and $G=NR$ is the product of $N$ and $R$.  Suppose further that $\Delta_R(r)=\Delta_G(r)$ for all $r\in R$. Then, if $\rho$ is a representation 
of $R$, then $\ind_{R\cap N}^N(\rho|_{R\cap N})\cong (\ind_R^G\rho)|_N$.
\end{lemma}
\begin{proof} Since $N$ is normal in $G$ and $N\cap R$ is normal in $R$, it follows from the assumptions that
$(\Delta_{G}(r)/\Delta_R(r))^\frac{1}{2}=1$ for all $r\in R$ and 
$(\Delta_{N}(r)/\Delta_{N\cap R}(r))^\frac{1}{2}=(\Delta_{G}(r)/\Delta_R(r))^\frac{1}{2}=1$ for $r\in N\cap R$.
It follows that the induced representation $\ind_R^G\rho$ acts on the Hilbert-space completion of 
\begin{equation}\label{eq-Frho}
\begin{split}
\mathcal F_\rho^R:=\{\xi\in C(NR, H_\rho): &\;\xi(xr)=\rho(r^{-1})\xi(x)\\
&\;\text{for all $x\in NR, r\in R$ and $\xi$ has compact support modulo $R$}\},
\end{split}
\end{equation}
while $\ind_{N\cap R}^N(\rho|_{N\cap R})$ 
acts on the Hilbert space completion of
\begin{equation}\label{eq-Frho1}
\begin{split}
\mathcal F_\rho^{N\cap R}:=&\{\xi\in C(N, H_\rho): \;\xi(xr)=\rho(r^{-1})\xi(x)\\
&\;\quad\text{for all $x\in N, r\in R\cap N$ and $\xi$ has compact support modulo $N\cap R$}\}
\end{split}
\end{equation}
It is then straightforward to check that we obtain a  bijective linear map 
$$\Phi: \mathcal F_\rho^{N\cap R}\to \mathcal F_\rho^{R}; \xi\mapsto\tilde\xi$$
with  $\tilde\xi(nr)=\rho(r^{-1})\xi(n)$.
It clear that this map preserves the left translation action by $N$ on both spaces. So the result will follow, if we can show that $\Phi$ preserves the inner products on both spaces.
For this let $\beta:N\to [0,\infty)$ be a Bruhat section for $N/(N\cap R)$ and let $\varphi\in C_c(R)^+$ such that 
$\int_R\varphi(r)\,dr=1$. It is then straightforward to check that
$$\tilde\beta:NR\to [0,\infty): \tilde\beta(nr)=\int_{N\cap R}\beta(nl)\varphi(l^{-1}r)\,dl$$
is a Bruhat section for $NR/R$. 

In order to compare Haar measures on $NR$ and $N$ we consider the semi-direct product $N\rtimes R$ given by the conjugation action of $R$ on $N$. Then one checks that
$$q:N\rtimes R\to NR; q(n,r)=nr$$
is a surjective homomorphism with $\ker\ q=\{(r^{-1},r): r\in N\cap R\}$ isomorphic to $N\cap R$ via projection on the second factor. Thus, using  Weil's integral formula we get 
$$\int_{N\rtimes R} f(n,r)\,dn\,dr=\int_{NR} \int_{N\cap N} f(nl^{-1}, l r)\,dl\, d(nr).$$
Using this, we compute for $\xi, \eta \in \mathcal F_\rho^R$:
\begin{eqnarray*}
\langle \xi, \eta \rangle & = & \int_{NR} \tilde{\beta}(nr)
\langle \xi(nr), \eta(nr) \rangle \;d(nr) \\
& = & \int_{NR} \int_{N \cap R} \beta(nl) \varphi(l^{-1}r)\; dl \;
\langle \xi(nr), \eta(nr) \rangle\; d(nr) \\
& = & \int_{N \rtimes R}
\beta(h) \varphi (r) \langle \xi(nr), \eta(nr) \rangle \; d(n,r)\\
& = & \int_N \int_R \beta(n) \langle \xi(n), \eta(n) \rangle
\varphi(r)\; dr \; dn
 =  \int_N \beta (n) \langle \xi(n), \eta(n) \rangle \; dn\\
& = & \langle \Phi(\xi), \Phi(\eta) \rangle.
\end{eqnarray*}
\end{proof}

Let us also  recall the Fell topology and the notion of weak containment for representations: In general, if $A$ is a $C^*$-algebra, then we denote by 
$\Rep(A)$ the collection of all equivalence classes of nondegenerate $*$-representations of $A$.
A base of the {\em Fell topology} on $\Rep(A)$ is given by all sets of the form
$$U(\pi, I_1,\ldots, I_l)=\{\rho\in \Rep(A): \rho(I_i)\neq \{0\}\}\;\text{for all $1\leq i\leq l$}\},$$
where $I_1,\ldots, I_l$ is any finite family of closed ideals in $A$. Restricted to $\widehat{A}$, this becomes the usual Jacobson topology on $\widehat{A}$. Note that $\Rep(A)$ only becomes a set, and hence a topological space, after restricting the dimensions of the representations by some cardinal $\kappa$, but we can always choose this cardinal big enough, so that all representations we are interesting in lie in $\Rep(A)$.

The Fell topology is closely related to the notion of {\em weak containment} of representations:
If $\pi\in \Rep(A)$ and $\Sigma\subseteq \Rep(A)$, then we say $\pi$ {\em is weakly contained in $\Sigma$} (written $\pi\prec\Sigma$), if $\ker \pi\supseteq \ker\Sigma:=\cap_{\sigma\in \Sigma} \ker\sigma$. Two subsets $\Sigma_1,\Sigma_2\subseteq \Rep(A)$ are called {\em weakly equivalent} (written $\Sigma_1\sim\Sigma_2$) if every element of $\Sigma_1$ is weakly contained in $\Sigma_2$ and vice versa. This is equivalent to $\ker\Sigma_1=\ker\Sigma_2$. Restricted to $\widehat{A}$, weak containment is same as the closure relation in $\widehat{A}$ with respect to the Jacobson topology. Note also that every $\pi\in \Rep(A)$ is weakly equivalent to its {\em spectrum} $\Sp(\pi):=\{\sigma\in \widehat{A}: \sigma\prec \pi\}$ and that $\Sigma\sim \bigoplus_{\sigma\in \Sigma}\sigma$ for every $\Sigma\subseteq \Rep(A)$.

\begin{remark}\label{rem-weak-containment}
The connection between weak containment and the Fell topology is given by the following facts
\\
{\bf (a)} A net $\pi_i$ converges to $\pi$ in $\Rep(A)$ if and only if
every subnet of $\pi_i$ weakly contains $\pi$  (see \cite[ Proposition 1.2]{Fel}).
\\
{\bf (b)} If $\pi_i$ converges to $\pi$ in $\Rep(A)$ and $\rho\prec \pi$, then $\pi_i$ converges to $\rho$ in $\Rep(A)$  (see {\cite[Proposition 1.3]{Fel}}).
\\
{\bf (c)}
Let $(\pi_i)_{i \in I}$ be a net in $\Rep(A)$ with $\pi_i
\rightarrow \pi$ for some $\pi \in \widehat{A}$. For every $i \in I$,
let $D_i$ be a dense subset of $\Sp(\pi_i)$. Then there exists a
subnet $(\pi_{\lambda})_{\lambda \in \Lambda}$ of $(\pi_i)_{i \in
I}$ and a net $(\rho_{\lambda})_{\lambda \in \Lambda}$ in
$\widehat{A}$ such that $\rho_{\lambda} \in D_{\lambda}$ for all
$\lambda \in \Lambda$ and $\rho_{\lambda} \rightarrow \pi$ in
$\widehat{A}$ (see {\cite[Theorem 2.2]{Sch}}) .
\end{remark}




If $G$ is a locally compact group, then the $C^*$-group algebra  $C^*(G)$ of $G$ is defined as the 
enveloping $C^*$-algebra of $L^1(G)$. If $\pi$ is a unitary representation of $G$, then $\pi$ integrates to a $*$-representation of $L^1(G)$ via $L^1(G)\ni f\to \pi(f)=\int_Gf(x)\pi(x)\,dx\in \B(H_\pi)$.
This representation extends uniquely to $C^*(G)$ and this procedure gives a bijection between
the set $\Rep(G)$ of all equivalence classes of unitary representations of $G$ and $\Rep(C^*(G))$.
The Fell topology  and the notion of weak containment  on $\Rep(G)$ are defined via identifying 
$\Rep(G)$ with $\Rep(C^*(G))$ in this way.

\begin{remark}\label{rem-group-weak-containment}
We list some useful facts about weak containment of group representations:
\\
{\bf (a)} For $\rho\in \Rep(G)$ and $\Sigma\subseteq \Rep(G)$ with $\rho\prec \Sigma$ we have $\rho\otimes \pi\prec \Sigma\otimes\pi:=\{\sigma\otimes \pi: \sigma\in \Sigma\}$ (see {\cite[Theorem 1]{Fel2}}).
\\
{\bf (b)} Suppose that $N$ is a closed normal subgroup of $G$ and $\rho\in \Rep(N)$. Then
$$(\ind_N^G\rho)|_N\sim \{x\cdot\rho: x\in G\},$$
where $x\cdot \rho(n)=\rho(x^{-1}nx)$ denotes the conjugate of $\rho$ by $x$. In particular, we get 
$\rho\prec (\ind_N^G\rho)|_N$.\\
{\bf (c)} Let $G$ be an amenable, locally compact group and let $H$ be a
closed subgroup of $G$. Then $\pi\prec\ind_H^G(\pi|_H)$ for all $\pi\in \Rep(G)$ 
(combine \cite[Theorem 5.1]{Gre} with part (b) of Remark \ref{rem-properties-of-induction}).
\end{remark}



%

Note that if $N$ is a normal subgroup of $G$, we have $\ind_H^G1_N\cong \lambda_{G/N}$, the left regular representation of $G/N$. If $G/N$ is amenable, then $\lambda_{G/N}$ is a faithful representation of $C^*(G/N)$ and hence weakly equivalent to $\widehat{G/N}$. Moreover, $\widehat{G/N}$ is weakly equivalent to $D$, for any dense subset $D$ of $\widehat{G/N}$. 
Thus, as a (well-known) corollary of the above, we get for all $\pi\in \Rep(G)$:
\begin{equation}\label{eq-ind-res}
\ind_N^G(\pi|_N)\cong \ind_N^G1_N\otimes \pi\cong \lambda_{G/N}\otimes\pi\sim \widehat{G/N}\otimes\pi\sim D\otimes \pi
\end{equation}
whenever $D$ is a dense subset of $\widehat{G/N}$.

By the pioneering work of Fell we also know that weak containment is preserved by induction and restriction of representations:

\begin{theorem}[{ \cite{Fell2}}]\label{thm-cont-ind-res}
Suppose that $H$ is a closed subgroup of $G$, $\Sigma\subseteq \Rep(H)$ and $\sigma \in \Rep(H)$.
Then 
$$\sigma\prec\Sigma\quad\Longrightarrow\quad \ind_H^G\sigma\prec\ind_H^G\Sigma:=\{\ind_H^G\tau : \tau\in \Sigma\}.$$
Similarly, if $\pi\in \Rep(G)$ and $\Pi\subseteq \Rep(G)$, then
$$\pi\prec \Pi\quad\Longrightarrow\quad \pi|_H\prec\Pi|_H:=\{\rho|_H: \rho\in \Pi\}.$$
\end{theorem}

\section{Some results of Howe and two-step nilpotent groups}\label{sec-Howe}

In this section we want to recall some general results on the representation theory of nilpotent locally compact groups which are mainly due to Howe in \cite{How}. Throughout this section suppose that $G$ is a $k$-step nilpotent locally compact group for some $k\in \NN$. 
We write
$$\{e\}=Z_0\subseteq Z_1\subseteq Z_2\subseteq\cdots\subseteq Z_k=G$$
for the ascending central series of $G$. We usually write $Z$ for the center $Z_1$ of $G$.
For two subsets $A,B\subseteq G$ we write $(A; B)$ for the 
subgroup generated by all commutators $(a;b):=aba^{-1}b^{-1}$ with $a\in A, b\in B$.

 If $k\geq 2$, let $A$ be any maximal abelian subgroup of $Z_2$ 
and let $N$ denote the centralizer of $A$ in $G$. Then $N$ has nilpotence length at most $k-1$ and we obtain a nondegenerate bihomomorphism
$$\Phi: G/N\times A/Z\to Z; (\dot x, \dot y)\mapsto (x;y)=xyx^{-1}y^{-1}.$$
Recall that nondegeneracy means that the corresponding homomorphisms
$$G/N\to \Hom(A/Z,Z); \dot x\mapsto \Phi(\dot x, \cdot)\quad\text{and}\quad 
A/Z\to \Hom(G/N, Z); \dot y\mapsto \Phi(\cdot, y)$$
are both injective.
Now, if $\psi\in \widehat{Z}$, we obtain a bicharacter
$$\Phi_\psi: G/N\times A/Z\to \TT;\; \Phi_\psi(\dot x, \dot y)=\psi\big((x;y)\big),$$
which is always nondegenerate if $\psi$ is faithful on ${(G;A)}\subseteq Z$. 
Note that if $\Phi_\psi$ is nondegenerate, then the corresponding injective homomorphisms 
\begin{equation}\label{eq-nondegenerate}
 G/N\to \widehat{A/Z}: \dot{x}\mapsto \Phi_\psi(\dot{x},\cdot)\quad\text{and}\quad
A/Z\to \widehat{G/N}:\dot{y}\mapsto\Phi_\psi(\cdot ,\dot{y}),
\end{equation}
have both dense range. To see this, suppose that $\Lambda\subseteq \widehat{A/Z}$ is the closure of 
the image of $G/N$ under the first homomorphism in (\ref{eq-nondegenerate}). Then $\Lambda=L^\perp$ for some closed subgroup $L$ of $A/Z$, and one easily checks that $L$ lies in the kernel of the second homomorphism in  (\ref{eq-nondegenerate}). Thus $L=\{e\}$ must be trivial and $\Lambda=\widehat{A/Z}$.

Let us also remind the reader that it follows from Schur's lemma, that for every irreducible representation of 
$G$ there exists a unique character $\psi_{\pi}\in \widehat{Z}$ such that $\pi|_Z=\psi_{\pi}\cdot 1_{H_\pi}$. 
Of course, if $\pi$ is faithful on ${(G;A)}\subseteq Z$, the same holds for  $\psi_{\pi}$. 
The following proposition is a slight reformulation of  \cite[Proposition 5]{How}. For completeness, we give a proof.

\begin{proposition}\label{A-N-proposition}
Suppose that $G$ is a  second countable nilpotent locally compact group. Let $Z, A$ and $N$ be as above and suppose that  $\pi\in \widehat{G}$ with central character $\psi=\psi_\pi$ such that 
$\pi$ is not one-dimensional and 
the bicharacter
$\Phi_\psi: G/N\times A/Z\to\TT$ is nondegenerate (which is automatic if $\pi$, and hence $\psi$,
is faithful on ${(G,A)}\subseteq Z$).
Then  there exists some $\rho\in \widehat{N}$ with $\rho|_Z\sim \psi$, 
$\ind_N^G\rho$ is irreducible and 
$\ker\pi=\ker(\ind_N^G\rho)$ in $C^*(G)$. 
\end{proposition}

We devide the argument into two lemmas:

\begin{lemma}\label{lemma1}
Let $\pi\in \widehat{G}$, $\psi\in \widehat{Z}$ and $N$ be as in the proposition. Then $\ker\pi=\ker(\ind_N^G(\pi|_N))$.
\end{lemma}
\begin{proof}
For each $y\in  A$ let $\chi_y=\Phi_\psi(\cdot, y)\in \widehat{G/N}$. Then
$\{\chi_y: y\in A\}$ is dense in $\widehat{G/N}$. For $y\in A$ and $x\in G$ we compute
$$\chi_y(x)\pi(x)=\psi(yxy^{-1}x^{-1})\pi(x)=\pi(yxy^{-1}x^{-1})\pi(x)=\pi(y)\pi(x)\pi(y^{-1})$$
so that $\chi_y\otimes \pi$ is unitarily equivalent to $\pi$.
Now, by Equation (\ref{eq-ind-res}) we have 
$$\ind_N^G(\pi|_N)\sim \widehat{G/N}\otimes \pi\sim \{\chi_y\otimes \pi:y\in A\}\sim \pi.$$
\end{proof}

If $A, N\subseteq G$ are as above, then the center of $N$ contains the center $Z$ of $G$ and therefore
the central character $\psi_\rho$ restricts to some character $\psi$ of $Z$.  We then get

\begin{lemma}\label{lemma2}
Suppose that $G, A$ and $N$ are as above and let $\rho\in \widehat{N}$ such that the central character $\psi_\rho$ of $\rho$ restricts to $\psi\in \widehat{Z}$ such that 
$\Phi_\psi:G/N\times A/Z\to\TT$ is nondegenerate.
Then $\pi:=\ind_N^G\rho$ is irreducible with central character $\psi_\pi=\psi$. 
\end{lemma}
\begin{proof}
Let $T\in \B(H_{\ind\rho})$ be any intertwiner for $\ind_N^G\rho$. We show that $T$ also intertwines the multiplication operators $P^\rho(\varphi):H_{\ind\rho}\to H_{\ind\rho}$, $\varphi\in C_0(G/N)$, given by 
$P^\rho(\varphi)\xi=\varphi\cdot\xi$. Since $\rho$ is irreducible, it will then follow from Mackey's imprimitivity theorem and Schur's lemma that $T$ is a multiple of  the identity, and hence that $\ind_N^G\rho$ is irreducible.

Since any $\varphi\in C_0(G/N)$  can be approximated in norm by Fourier-transforms of integrable functions  on 
$\widehat{G/N}$, it suffices to show that $T$ intertwines 
the multiplication operators $P^\rho(\chi)$ for all $\chi\in \widehat{G/N}$. By density, this will
follow if  $T$ intertwines all operators $P^\rho(\chi_y)$,  $y\in A$, with $\chi_y=\Phi_\psi(\cdot, \dot{y})$.
To see that this is the case, observe first that $\rho|_Z=\psi_{\pi}\cdot 1_{H_\rho}$. Indeed, the restriction of $\rho$ to the center $Z$ of $G$ coincides with the restriction of $x\cdot\rho$ to $Z$ for 
all $x\in G$. Since restriction preserves weak containment, this implies that 
$$\rho|_Z\sim \{(x\cdot\rho)|_Z: x\in G\}\sim \pi|_Z$$
from which the claim follows.
Now,  for $y\in A$, $x\in G$ and $\xi\in H_{\ind\rho}$  we get
$$
\rho(y^{-1})\big(\ind_N^G\rho(y)\xi\big)(x)=\rho(y^{-1})\xi(y^{-1}x)=\xi(xx^{-1}y^{-1}xy)=\rho(y^{-1}x^{-1}yx)\xi(x)=
\chi_y(x)\xi(x).$$
Since $A$ lies in the center of $N$, we have $\rho(y^{-1})=\psi_\rho(y^{-1})\id_{H_\rho}$ is a multiple of the identity. Thus the above computation implies 
$P^\rho(\chi_y)=\psi_\rho(y^{-1})\cdot \ind_N^G\rho(y)$ and hence $T$  intertwines $P^\rho(\chi_y)$ for all $y\in A$.

The last assertion follows from the fact that $\rho\prec\{x\cdot\rho:x\in G\}\sim \pi|N$, and hence $\rho|_Z\prec \pi|_Z$. 
\end{proof}

Above and in the following proof of the proposition we use the fact that for two characters $\chi,\mu$ of an abelian locally compact group $C$ we have $\chi\prec \mu\;\Leftrightarrow\; \chi=\mu$. It follows from this that if $\pi\in \widehat{G}$ and $\psi\in \widehat{C}$ for some central closed subgroup $C$ of $G$,
then:
$$\text{$\psi$ is the central character of $\pi$ on $C$}\quad\Leftrightarrow\quad
\psi\sim\pi|_C\quad\Leftrightarrow\quad
\psi\prec \pi|_C\quad\Leftrightarrow\quad
\pi|_C\prec \psi.$$

\begin{proof}[Proof of Proposition \ref{A-N-proposition}]
Since $G$ is second countable, it follows from \cite[Theorem 2.1]{Goo} that there exists some 
$\rho\in \widehat{N}$ with $\pi|_N\sim \{x\cdot\rho: x\in G\}$. This implies that $\rho|_Z\prec\pi|_Z\sim \psi$.
 It follows then from Lemma \ref{lemma2}
that $\ind_N^G\rho$ is irreducible. Since $\ind_N^G(x\cdot\rho)\cong\ind_N^G\rho$ for all $x\in G$ we also get 
$$\ind_N^G\rho\sim\ind_N^G \{x\cdot\rho:g\in G\}\sim \ind_N^G(\pi|_N)\sim\pi$$
by Lemma \ref{lemma1}. 
\end{proof}

\begin{remark}  We should also note that 
Proposition \ref{A-N-proposition}  together with an easy induction argument on the nilpotence length of $G$ implies that for any primitive ideal $P\in \Prim(G)$  there exists a closed subgroup $H$ of $G$ and a character $\mu$ of $H$ such that  $\ker(\ind_H^G\mu)=P$. We refer to the original argument in 
\cite[Proposition 5]{How}  for the details.
\end{remark}

If $G$ is a locally compact group with center $Z$, then it follows from the fact that restriction of representations preserves weak containment that the central character $\psi_\pi\in \widehat{Z}$ for any $\pi\in \widehat{G}$ 
only depends on the kernel $\ker\pi\in \Prim(G)$. Thus we may write $\psi_P$ instead of $\psi_\pi$ if $P=\ker\pi$ and we obtain a natural decomposition as a disjoint union
\begin{equation}\label{prim-decom}
\Prim(G)=\bigcup_{\psi\in \widehat{Z}} \mathcal P_\psi,\quad\text{with}\quad
\mathcal P_\psi=\{P\in \Prim(G): \psi_P=\psi\}\quad \text{for}\; \psi\in \widehat{Z}.
\end{equation}
We now want to study the sets $\mathcal P_\psi$ more closely in case where $G$ is two-step nilpotent.
For this we want to introduce some notation: If $\psi\in \widehat{Z}$, then a closed subgroup $A\subseteq G$ 
is called {\em subordinate to $\psi$} if 
$$\psi\big((x;y)\big)=1\quad\text{for all}\; x,y\in G.$$
If $A_\psi$ is  maximal with this property, then $A_\psi$ is called a polarizing subgroup for $\psi$.
Note that $A_\psi$ always contains $Z$ and it is not difficult to check that $A_\psi$ is a polarizing subgroup for $\psi\in \widehat{Z}$ if and only if $A_\psi/K_\psi$ is a maximal abelian subgroup of $G/K_\psi$,  if 
$K_\psi$ denotes the group kernel of $\psi$ in $Z$. Moreover, if we define the {\em symmetrizer} 
$Z_\psi$ of $\psi$ as
$$Z_\psi=\{x\in G: \psi\big((x;y)\big)=1\;\text{for all}\; y\in G\},$$
then $Z_\psi/K_\psi$ is the center of $G/K_\psi$.
The following proposition is well-known (e.g. see \cite[Lemma 2]{Kan-primal} or \cite{Pogi}). However, for the readers convenience, we include a proof.

\begin{proposition}\label{prop-two-step}
Let $G$ be a two-step nilpotent locally compact group and for $\psi\in \widehat{Z}$ let  $K_\psi, Z_\psi$ as above and let $A_\psi\subseteq G$ be any $\psi$-polarizing subgroup for $\psi$. 
Let $\mathcal Z_\psi:=\{\chi\in \widehat{Z}_\psi: \chi|_Z=\psi\}$ and 
$\mathcal A_\psi=\{\mu\in \widehat{A}_\psi: \mu|_Z=\psi\}$. Then the following are true
\begin{enumerate}
\item The map $\ind_\psi: \mathcal Z_\psi:\to \mathcal P_\psi; \chi\mapsto \ker(\ind_{Z_\lambda}^G\chi)$ is a well defined homeomorphism.
\item  $\ind_{A_\psi}^G\mu$ is irreducible for all $\mu\in \mathcal A_\psi$ and 
$\ker(\ind_{A_\psi}^G\mu)=\ker\big(\ind_{Z_\psi}^G(\mu|_{Z_\lambda})\big).$
\end{enumerate}
\end{proposition}

As a direct corollary of this result we get

\begin{corollary}\label{cor-polarizing}
Let $G$ be a two-step nilpotent locally compact group and  let $A_\psi$ and $B_\psi$ be two polarizing subgroups for some given $\psi \in \widehat{Z}$. Let $\chi\in \widehat{A}_\psi$ and $\mu\in \widehat{B}_\psi$
such that $\chi|_{Z_\psi}=\mu|_{Z_\psi}$. Then $\ind_{A_\psi}^G\chi\sim\ind_{B_\psi}^G\mu$.
\end{corollary}

The proof  of Proposition \ref{prop-two-step} will follow from 

\begin{lemma}\label{lemma3}
Suppose $G$ is a two-step nilpotent locally compact group and suppose that  $C$ is a closed subgroup of $Z=Z(G)$ which contains the commutator subgroup $(G;G)$. Let $\psi\in \widehat{C}$ be a faithful character of $C$ and let $A$ be a maximal abelian subgroup of $G$.
Let $\mathcal P_\psi$ denote the set of primitive ideals of $G$ with central character $\psi$ on $C$ (i.e., we have $\psi=\psi_P|_C$ if $\psi_P\in \widehat{Z}$ denotes the central character of $P$) and let 
$\mathcal Z_\psi=\{\chi\in \widehat{Z}: \chi|_C=\psi\}$ and 
$\mathcal A_\psi=\{\mu\in \widehat{A}: \mu|_C=\psi\}$. Then 
\begin{enumerate}
\item The map $\ind_\psi: \mathcal Z_\psi:\to \mathcal P_\psi; \chi\mapsto \ker(\ind_{Z}^G\chi)$ is a well defined homeomorphism.
\item  $\ind_{A}^G\mu$ is irreducible  and
$\ker(\ind_{A}^G\mu)=\ker\big(\ind_{Z}^G(\mu|_{Z_\lambda})\big)$ for all $\mu\in \mathcal A_\psi$.
\end{enumerate}
\end{lemma}
\begin{proof} Since $G$ is two-step nilpotent, it follows that $A$ coincides with its centralizer $N$ in $G$.
Since $\psi$ is faithful on ${(G;G)}$, the same is true for any character $\mu\in \mathcal A_\psi$. Thus, it follows then from Lemma \ref{lemma2} that 
$\ind_A^G\mu$ is irreducible for all such $\mu$ with $\ker(\ind_A^G\mu)\in \mathcal P_\psi$.
Let $\chi:=\mu|_Z$ and let $\mathcal A_\chi=\{\nu\in \widehat{A}: \nu|_Z=\chi\}$. Then 
$\ind_Z^A\chi=\ind_Z^A(\mu|_Z)\sim \widehat{A/Z}\otimes\mu=\mathcal A_\chi$.
Since $\chi$ is faithful on ${(G;G)}$ we also see that the map
$G/A\to \widehat{A/Z}; \dot{x}\mapsto \mu_x$ with 
$\mu_x(y)=\chi\big((x;y)\big)$ has dense range in $\widehat{A/Z}$.
For $x\in G$ and $y\in A$ we get
\begin{align*}
(x\cdot\mu\otimes\bar{\mu})(y)=\mu(xyx^{-1})\overline{\mu(y)}=\mu(xyx^{-1}y^{-1})=\chi(xyx^{-1}y^{-1})=\mu_x(y)
\end{align*}
from which it follows that the orbit $\{x\cdot\mu: x\in G\}=\{\mu_x\otimes \mu :x\in G\}$ is dense in 
$\widehat{A/Z}\otimes\mu=\mathcal A_\chi\sim \pi|_A$.
Thus we get
$$\ind_Z^G\chi=\ind_A^G(\ind_Z^A\chi)\sim \ind_A^G\mathcal A_\chi\sim \ind_A^G\{x\cdot\mu: x\in G\}
\sim\ind_A^G\mu,$$
since $\ind_A^G(x\cdot\mu)\cong \ind_A^G\mu$ for all $x\in G$.

Combining the above results we get (2) and we see that the map in (1) is a well defined injective map.
To see that it is surjective, it suffices to show that for each $\pi\in \widehat{G}$ with central character 
$\chi\in \mathcal Z_\psi$ and for each $\mu\in \mathcal A_\chi$ we have $\pi|_A\sim \{x\cdot \mu: x\in G\}$,
since it follows then from Lemma \ref{lemma1} that
$$\pi\sim \ind_A^G(\pi|_A)\sim \ind_A^G\{x\cdot \mu: x\in G\}\sim \ind_A^G\mu.$$
By what we saw above, the desired result is equivalent to the statement $\pi|_A\sim \mathcal A_\chi$. 
To see that this is true note first that for any fixed $\mu\in \mathcal A_\chi$ we have
$$
\pi|_A\prec \ind_Z^A(\pi|_Z)\sim \ind_Z^A\chi\sim \ind_Z^A(\mu|_Z)\\
\sim \lambda_{A/Z}\otimes\mu\sim \widehat{A/Z}\otimes\mu\sim\mathcal{A}_\chi
$$
On the other hand, if $\mu\in \widehat{A}$ is weakly contained in $\pi|_A$, then $\mu|_Z$ is weakly contained in $\pi|_Z\sim \chi$, hence $\mu|_Z=\chi$. This proves the claim.

It follows now that the map in (1) is bijective with inverse given by mapping any $P\in \mathcal P_\psi$ to its central character $\chi_P\in \widehat{Z}$. Since induction and restriction preserve weak containment, 
the map in (1) and its inverse are both continuous, thus a homeomorphism.

\end{proof}

\begin{proof}[Proof of Proposition \ref{prop-two-step}]
The proposition follows from Lemma \ref{lemma3} by passing from $G$ to $G/K_\psi$ for fixed $\psi\in \widehat{Z}$, in which case the groups  $\dot C_\psi:=Z/K_\psi$, $\dot Z_\psi$ and $\dot A_\psi$ will play the
 r\^oles of $C, Z$ and $A$ in Lemma \ref{lemma3}.
\end{proof}

We finish this section with the following lemma, which will be used later:

\begin{lemma}\label{lem-injective-pairing}
Let $G$ be a two-step nilpotent locally compact group with center $Z$ and let $\psi\in \widehat{Z}$. 
Let $A\subseteq G$ be a closed subgroup of $G$ which is subordinate to $\psi$. Then the following are equivalent:
\begin{enumerate}
\item $A$ is a polarizing subgroup for $\psi$.
\item The homomorphism $\Phi_\psi: G/A\to \widehat{A}; \Phi_\psi(\dot x)(y)=\psi\big((x;y)\big)$ is injective.
\end{enumerate}
\end{lemma}
\begin{proof}
(1) $\Rightarrow$ (2) was already observed in the proof of the previous lemma.
To see the converse, assume that $L$ is a closed subgroup of  $G$ which is subordinate to $\psi$ and which properly contains $A$. Then $\psi\big((x;y)\big)=1$ for all $a\in L, y\in A$, since $A\subseteq L$. Hence 
$L/A$ lies in the kernel of $\Phi_\psi$.
\end{proof}

\section{Nilpotent Lie pairs}\label{section-Lie-pair-construction}

In this section we want to give our notion of Lie-pairs $(G,\G)$, which we want to use to attack the problem of a more general Kirillov theory for locally compact nilpotent groups. The conditions we use here are much stronger than the condition of an ``elementary exponentiable'' group as introduced by Howe in \cite{How}. However, we want to give a set of conditions which allows to prove a version of Kirillov's orbit method without (almost) no further restrictions on the structure of the underlying groups, but which is general enough to cover a large class of examples. In order to be successful, we need a notion of Lie pairs which is stable under passing to certain characteristic subgroups/algebras and quotients. 
 
Recall that if $R$ is a commutative Ring with unit, then 
an algebra $\G$ over $R$
is called a Lie algebra over $R$ if its multiplication, denoted
by $(X,Y) \mapsto [X,Y]$, satisfies the following identities:
\begin{itemize}
\item[(1)] $[X,X]=0$ and
\item[(2)] $
[[X,Y],Z]+[[Y,Z],X]+[[Z,X],Y]=0.$ (Jacobi-identity)
\end{itemize}
A topological Lie algebra over $R$ is a Lie algebra over $R$ with
a Hausdorff topology such that the Lie algebra operations $X
\mapsto -X$, $(X,Y) \mapsto X+Y$, and $(X,Y) \mapsto [X,Y]$ are
continuous with respect to this topology. We say that a Lie algebra $\G$ over $R$ is 
{\em nilpotent of length $l\in \NN$}, if every 
$l$-fold commutator is zero and $l$ is the smallest positive integer with this property.

\begin{definition}\label{definition-k-Lie-pair}
For $k \in \N$ or $k=\infty$, let $\Lambda_k$ denote the smallest subring of $\Q$  in which 
every prime number $p\leq k$ is invertible, i.e.,  $\Lambda_k:=\Z[\frac{1}{k!}]$ if 
$k<\infty$ and $\Lambda_\infty=\QQ$. Let $G$ be a locally compact, second countable
group  and let $\G$ be a
nilpotent topological Lie algebra over $\Z$ with nilpotence length $l \leq k$.
We then call the pair $(G,\G)$ a
{\em nilpotent $k$-Lie pair} (of nilpotence length $l$) if the following properties are satisfied
\begin{itemize}
\item[(i)] The additive group $\G$ is a
$\Lambda_k$-module, extending the $\Z$-module structure of $\G$.
\item [(ii)] There exists a homeomorphism $\exp:\G
\rightarrow G$, with inverse denoted by $\log$, satisfying the
Campbell-Hausdorff formula (see the appendix for details on this formula).
\end{itemize} 
\end{definition}

Of course, any simply connected real Lie group $G$ with corresponding Lie-algebra $\G$ 
forms an $\infty$-Lie-pair $(G,\G)$ with respect to the ordinary exponential map $\exp:\G\to G$.
The same holds for any unipotent group over the $p$-adic numbers $\Q_p$. On the other extreme,
every abelian second countable locally compact group $G$ gives rise to a $1$-Lie-pair
$(G,\G)$ with $\G=G$ and $\exp=\id:G\to G$. We shall see more interesting examples in the final section of this paper.

 \begin{remark}\label{remark-f-in-Hom(G,W)-is-Lambda_k-linear} 
 {\bf (1)} If   $\G$ and
$\mathfrak{m}$  are $\Lambda_k$-modules and  $f \in
\Hom(\G,\mathfrak{m})$ is a group homomorphism, then it is an exercise to check that
$f$ is automatically $\Lambda_k$-linear.  In particular, it follows that 
for a Lie algebra $\G$ over $\Z$ such
that the additive group $\G$ is a $\Lambda_k$-module the commutator map $[.,.]: \G
\times \G \rightarrow \G$ is $\Lambda_k$-bilinear, thus $\G$ is a Lie algebra over $\Lambda_k$.
\\
{\bf (2)} If $(G,\G)$ is a nilpotent $k$-Lie pair and $x=\exp(X)$ for some $X\in \G$, then 
for every $m\in \NN$ with $m\leq k$ 
 we may define the (unique)``$m$th" root of $x$ as
$
x^{\frac{1}{m}}:= \exp(\frac{1}{m}X).
$
More generally, if  $\lambda=\frac{n}{m} \in \Lambda_k$  we may define
\[
x^{\lambda}:=\exp(\lambda X),
\]
which then implies that $\log(x^\lambda)=\lambda\log(x)$ for all $x\in G$ and $\lambda\in \Lambda_k$.
\end{remark}

\begin{definition}
Let $(G, \G)$ be a nilpotent $k$-Lie pair.
By a subalgebra of $\G$
we understand a Lie-subalgebra of $\G$, which is also a
$\Lambda_k$-module. An ideal of $\G$ is a subalgebra
$\n$ of $\G$ which has the additional property that
$[X,Y] \in \n$ for all $X \in \G$ and for all $Y \in
\n$.
\end{definition}

If $(G, \G)$ is a nilpotent $k$-Lie pair  and
if $\mathfrak{h}$ is a subalgebra of $\G$, then it follows
directly from the Campbell-Hausdorff formula that
$H=\exp(\mathfrak{h})$ is a subgroup of $G$ and hence that $(\h, H)$ 
is again a $k$-Lie-pair.  We need to answer the question, which closed subgroups 
of $G$ appear in this way.
The following  theorem is  related to Howe's \cite[Proposition 3]{How}.


\begin{theorem}\label{k-complete-subgroups-give-subalgebras}
Let $k \in \N\cup\{\infty\}$ and let $(G, \G)$ be a nilpotent $k$-Lie pair. Then, for a closed subgroup $H$ of $G$ the following are equivalent:
\begin{enumerate}
\item $H$ is {\em exponentiable}, i.e., $\h=\log(H)$ is a subalgebra of $\G$.
\item $H$ is {\em $k$-complete} in the sense 
that $x^{\lambda} \in H$ for all $x \in H$ and for all $\lambda \in\Lambda_k$.
\end{enumerate}
\end{theorem}

Note that the direction (1) $\Rightarrow$ (2) follows directly from the equation $x^\lambda=\exp(\lambda X)$ for $\lambda\in \Lambda_k$.
The proof of the converse direction is given in the appendix (see Lemma \ref{lemma-k-complete-subgroup->subalgebra}). There we also formulate a number of important consequences which imply that certain characteristic subgroups of $G$ are exponentiable.

In order to get good results for the representation theory of $G$, we need some additional structure 
which allows to have an analogue of the linear dual $\G^*=\Hom(\G,\R)$ 
in the case of a real Lie algebra $\G$. Note that if we consider a real Lie algebra $\G$ simply
as a real vector space, then the linear dual  $\G^*=\Hom(\G,\RR)$ can be identified with the 
Pontrjagin dual $\widehat{\G}=\Hom(\G, \TT)$ of the underlying abelian group via the map
$$\Phi: \Hom(\G,\R)\to \widehat{\G}; f\mapsto \epsilon\circ f,$$
where $\epsilon:\R\to \TT$ is the basic character $\epsilon(x)=e^{2\pi i x}$. 
Using this isomorphism, it is possible to exploit the linear structure of $\G$ for the study of the 
representations of the corresponding Lie group $G$.  We now want to introduce a substitute for 
the pair $(\R,\epsilon)$ in more general situations:

\begin{definition}\label{def-dualizable}
Let $(G,\G)$ be a $k$-Lie-pair for some $k\in \N\cup \{\infty\}$. Suppose that
$\mathfrak{m}$ is a second countable locally compact $\Lambda_k$-module and 
 $\epsilon: \mathfrak{m}\to \TT$ is a character such that 
\begin{itemize}
\item[(a)] the kernel $\ker\epsilon\subseteq \mathfrak{m}$ does not contain any nontrivial
$\Lambda_k$-submodule of $\mathfrak{m}$.
\item[(b)] The map
\[
\Phi: \Hom(\G,\mathfrak{m}) \rightarrow \widehat{\G}, \; f \mapsto
\epsilon \circ f
\]
is an isomorphism of groups, where $\Hom(\G,\mathfrak{m})$
denotes the continuous group homomorphisms from $\G$ to
$\mathfrak{m}$ and $\widehat{\G}$ denotes the Pontrjagin dual of
the abelian group $\G$.
\item[(c)] For every closed $\Lambda_k$-subalgebra $\mathfrak{h}$ of $\G$
and for any $f \in \Hom(\mathfrak{h},\mathfrak{m})$ there exists
a map $\tilde{f} \in \Hom(\G,\mathfrak{m})$ such that
$\tilde{f}|_{\mathfrak{h}}=f$.
\end{itemize}
Then we call  $(G,\G)$ an {\em $(\mathfrak{m},\epsilon)$-dualizable nilpotent $k$-Lie pair} (or just a {\em dualizable nilpotent $k$-Lie pair} if confusion seems unlikely). 
We write
$\G^*:=\Hom(\G,\m)$
and we equip $\G^*$ with the compact open topology. 
\end{definition}

\begin{remark} 
Since $\G$ and $\m$ are second countable locally compact groups, the same is true 
for $\G^*$ and  the 
isomorphism $\Phi:\G^*\to\widehat{G}; \;\Phi(f)=\epsilon\circ f$ is a continuous isomorphism of groups, 
hence a topological isomorphism by the open mapping theorem for second countable locally compact groups.
Note also that it follows from Remark \ref{remark-f-in-Hom(G,W)-is-Lambda_k-linear} that any 
$f\in \Hom(\G,\m)$ is automatically a $\Lambda_k$-module map.
\end{remark}

\begin{lemma}\label{lemma-subalgebras-give-lie-pairs}
let $(G, \G)$ be an $(\mathfrak{m},\epsilon)$-dualizable nilpotent $k$-Lie pair. 
Then so are $(H,\h)$, if $\h$ is a closed subalgebra
 of $\G$  and $H=\log(\h)$, and $(G/N,\G/\n)$
if $\n$ is a closed ideal in $\G$ and $N=\exp(\n)$.
\end{lemma}

\begin{proof} We omit the straightforward proof.\end{proof}

\section{The Kirollov map}\label{section-Kirillov-map}

In this section we shall always assume 
that $(G, \G)$ is an $(\m,\epsilon)$-dualizable  nilpotent
$k$-Lie pair as in Definition \ref{def-dualizable}. 

\begin{definition}
Let $f \in \G^*$. A closed
subalgebra $\mathfrak{r}$ of $\G$ is said to be {\em $f$-subordinate} if
\[
f([\mathfrak{r}, \mathfrak{r}])= \{0\}.
\]
If $\mathfrak{r}$ is maximal with this property we say that
$\mathfrak{r}$ is a {\em polarizing subalgebra} for $f$.
\end{definition}


\begin{remark}\label{character->representation}
{\bf (a)}  Let $\mathfrak{r}$ be an $f$-subordinate subalgebra of $\G$
 and let $R=\exp{\r}$. It follows then from the Campbell-Hausdorff formula and the 
fact that $f$ vanishes on commutators in $\r$ that
$\varphi_{f}:R\to \TT$, defined by
\begin{equation}\label{definition-varphi-f}
\varphi_{f}(\exp X):= \epsilon(f(X)) \quad \mbox{for all}\; X \in
\mathfrak{r},
\end{equation}
is a unitary character of  $R$. 
In what follows we shall always stick to the notation $\varphi_f$ for this character, even when viewed on different polarizing subalgebras at the same time!
\medskip
\\
{\bf (b)} By an easy application of Zorn's lemma we see that for every subalgebra $\h\subseteq \G$ which is subordinate to some $f\in \G^*$, there exists a polarizing subalgebra $\r$ for $f$ which contains $\h$.
But we shall later see that for most of our constructions we need polarizing subalgebras with special properties,  which we shall introduce in Remark \ref{rem-standard} below.
Note that a given element $f \in \G^*$ may have different,
non-isomorphic polarizing subalgebras. 
%

\end{remark}

If $(G,\G)$ is a $k$-Lie pair with nilpotence length $l\leq k$, then the {\em adjoint action} of $G$ on $\G$ is defined by
the homomorphism 
$$\Ad:G\to \GL(\G); \Ad(x)(Y):=\log(x\exp(Y)x^{-1}).$$
If $x=\exp(X)$ for some $X\in \G$, then it follows from standard computations  that
\begin{equation}\label{eq-Ad}
\Ad(\exp(X))(Y)=\exp(\ad(X))(Y)=\sum_{n=0}^l\frac{1}{n!}\ad(X)^n(Y).
\end{equation}

Now if $(G,\G)$ is an $(\m,\epsilon)$-dualizable $k$-Lie pair, then the coadjoint action of $G$ on $\G^*$ is given, as usual, by
\begin{equation}\label{eq-coadjoint}
\Ad^*(x)(f)=f\circ \Ad(x^{-1})
\end{equation}

We are now going to  provide a ``standard way'' to recursively construct a polarizing subalgebra
 for a given $f\in \G^*$ with certain nice properties. We start with a lemma which allows to assume that 
$f$ is faithful on $\z(\G)$:

\begin{lemma}\label{largest-ideal-in-kernel-of-character}
Let $(G, \G)$ be an $(\m,\epsilon)$-dualizable nilpotent $k$-Lie pair, and let
$f \in \G^*$. Then there exists a largest ideal $\mathfrak{j}$ inside
the kernel of $f$. Let $q:\G\to \G/\j$ denote the quotient map and let 
$\tilde{f}\in (\G/\j)^*$ be defined by $\tilde{f}(q(X))=f(X)$. Then $\tilde{f}$ is faithful
on  $\z(\G/\mathfrak{j})$ and if $\mathfrak{r}\subseteq  \G$ is a subalgeba of $\G$, then $\r$ is
polarizing for $f$ if and only if
$\tilde{\mathfrak{r}}:= q(\mathfrak{r})$ is a polarizing
subalgebra for $\tilde{f}$.
\end{lemma}

\begin{proof}
 Zorn's Lemma assures the existence of a
maximal ideal $\j$ in $\ker(f)$ while uniqueness is guaranteed by the fact that if
$\mathfrak{j_1},\mathfrak{j_2}$ are two such, then their sum would
also be one, so that $\j$ is indeed the largest ideal in $\ker(f)$. 
Note that $\j$ is automatically closed in $\G$.
Let $\tilde{f}\in (\G/\j)^*$ be as in the lemma. If it is not faithful on $\z(\G/\j)$, then 
$\tilde{\k}=\ker{\tilde{f}}\cap \z(\G/\j)$ is a nontrivial closed ideal in $\G/\j$ and 
its inverse image $\k$ in $\G$ is a closed ideal of $\G$ lying in $\ker f$ and strictly larger than 
$\j$, a contradiction.
\end{proof}

In a similar fashion we get the following lemma. 

\begin{lemma}\label{lemma-maximal-ideal-inside-kernel-pi-and-f}
Let $(G,\G)$ be an $(\m,\epsilon)$-dualizable nilpotent $k$-Lie pair, and let
$\pi \in \widehat{G}$. Then there 
exists a maximal exponentiable normal subgroup
$J$ of $G$ such that $J \subseteq \ker(\pi)$.
\end{lemma}

\begin{proof}
To prove part $(i)$ we define
\[
\mathcal{M}:=\{I \mid I \; \hbox{is a normal exponentiable
subgroup of}\; G \; \hbox{and} \; I \subseteq \ker(\pi) \}.
\]
The set $\mathcal{M}$ is partially ordered by inclusion and since
$\{1_G \} \in \mathcal{M}$, it follows that $\mathcal{M}$ is
nonempty. Let $\mathcal{K}$ be a chain in $\mathcal{M}$. We claim
that the set
\[
J:= \overline{\bigcup_{I \in \mathcal{K}} I}
\]
is an upper bound for $\mathcal{K}$. To see this it suffices to show that $J$ is an
exponentiable subgroup of $G$. But this follows from the fact that 
$\log(J)=\overline{\cup_{I\in \mathcal{K}}\log{I}}$ is an ideal of $\G$, since 
$\log(I)$ is an ideal of $\G$ for all $I\in \mathcal{K}$.
\end{proof}

Recall that for any $\pi\in \widehat{G}$ the central character of $\pi$ is the unique character $\psi_\pi$ of the center $Z$ of $G$ such that $\pi|_Z=\psi_\pi\cdot 1_{H_\pi}$. 
 
\begin{lemma}\label{lemma-faithful-central-character}
Suppose that $(G,\G)$ is an $(\m,\epsilon)$-dualizable $k$-Lie pair and let $\pi\in \widehat{G}$ such that 
the group kernel of $\pi$ does not contain any non-trivial exponentiable subgroup. 
Let $f\in \G^*$ such that the central character $\psi_\pi$ equals $\varphi_f\in \widehat{Z}$.
Then $f$ is faithful on $\z=\z(\G)$.
\end{lemma}
\begin{proof}
This follows directly from the fact that $\z\cap \ker f$ is an ideal in $\G$.
\end{proof}

\begin{remark}\label{remark-nondegenerate}
Let $(G,\G)$ be an $(\m,\epsilon)$-dualizable $k$-Lie pair, let $A\subseteq Z_2(G)$ be a choice of a maximal abelian subgroup of $Z_2(G)$ and let $N\subseteq G$ denote the centralizer of $A$ in $G$. We checked  in Lemma \ref{lemma-log(A)-is-a-subalgebra} that 
$\a=\log(A)$ is a maximal abelian subalgebra of $\z_2(\G)$ and $\n=\log(N)$ is the centralizer of $\a$ in $\G$. Let $f\in \G^*$ be faithful on $\z=\z(\G)$. Then one easily checks that the bihomomorphism
\begin{equation}\label{eq-nondegenerate1}
\Psi_f:\G/\n\times\a/\z\to \m; \Psi_f(\dot X, \dot Y)=f([X,Y])
\end{equation}
is nondegenerate. This implies that the 
corresponding homomorphisms 
\begin{equation} \label{eq-nondegenerate2}
\Psi_{\G/\n}:\G/\n\to (\a/\z)^*; \dot X\mapsto \Psi_f(\dot X, \cdot\,)\quad \text{and}\quad 
\Psi_{\a/\z}: \a/\z\to (\G/\n)^*; \dot Y\mapsto \Psi_f(\,\cdot , \dot Y)
\end{equation}
are injective with dense range. Injectivity is clear and the fact that they have dense ranges follows from 
the fact  that via the identifications $(\a/\z)^*\cong \widehat{\a/\z}$ and $(\G/\n)^*\cong \widehat{\G/\n}$ given by $ g\mapsto \epsilon\circ g$ the maps in (\ref{eq-nondegenerate2}) can be identified with the 
homomorphisms  
$\G/\mathfrak{n}\to  \widehat{\mathfrak{a}/\mathfrak{z}}$ and 
$\mathfrak{a}/\mathfrak{z}\to \widehat{\G/\mathfrak{n}}$  corresponding to the nondegenerate 
bicharacter 
 $\epsilon\circ \Psi_f: \G/\n\times\a/\z\to \TT$. 
 
Now a short exercise, using the Campbell-Hausdorff formula and the fact that $[\G,\a]\subseteq \z$,
shows that for $X\in \G$, $Y\in \a$
and $x=\exp(X)$, $y\in \exp(Y)$ we have $\exp([X,Y])=(x;y)$, the commutator of $x$ and $y$ in $G$.
Thus, Identifying $\a$ with $A$ and $\n$ with $N$ via the exponential map, the 
pairing $\epsilon\circ \Psi_f$ is identified with 
\begin{equation} \label{eq-nondegenerate3}
\Phi_f:G/N\times A/Z\to\TT; \;\Phi_f(\dot x,\dot y)=\varphi_f\big((x;y)\big),
\end{equation}
which shows that this is also a nondegenerate bicharacter. 

\end{remark}


 Recall that if a group $G$ acts on a topological space $X$, then two elements $x,y\in X$ lie in the same $G$-quasi orbit, if $x\in \overline{G\cdot y}$ and $y\in \overline{G\cdot x}$. This determines an equivalence relation on $X$ and the equivalence class  $\mathcal O^G(x)$ of $x$ is called the {\em $G$-quasi-orbit of $x$.}

In what follows, if $\h\subseteq \G$ is any subset of $\G$, 
then $\h^\perp:=\{g\in \G^*: g(\h)=\{0\}\}$ denotes the annihilator of $\h$ in $\G^*$. 

\begin{lemma}[{cf. \cite[Lemma 11]{How}}]\label{lem-polarization}
Let $A, N\subseteq G$ as above, and  let $f\in \G^*$ such that $f$ is  faithful on $\z=\z(\G)$. Suppose that
 $\r\subseteq \n=\log(N)$ is a polarizing subalgebra for $f|_\n\in \n^*$ and let $R=\exp(\r)$. Then $\r$ is a polarizing subalgebra for $f$. Moreover, the following are true:
 \begin{enumerate}
\item  If for all $y\in N$ we have $\Ad^*(y)f|_\n\in (f+\r^\perp)|_\n \Leftrightarrow y\in R$,
then  for all $x\in G$ we have $\Ad^*(x)f\in f+\r^\perp\Leftrightarrow x\in R$.
 \item If $(f+\r^{\perp})|_\n\subseteq \overline{\Ad^*(R)f|_\n}$, then we also have 
 $f+\r^\perp\subseteq \overline{\Ad^*(R)f}$. 
 \item Suppose that $\r$ is a polarization for $(f+h)|_\n$ for all $h\in \r^\perp$ and 
that $(f+\r^{\perp})|_\n=\mathcal O^R(f|_\n)$ for the coadjoint action of $R$ on $\n^*$. 
Then $\r$ is a polarization for $f+h$ for all $h\in \r^\perp$ and $f+\r^\perp=\mathcal O^R(f)$ for the coadjoint action of $R$ on $\G^*$.
 \end{enumerate}
\end{lemma}
\begin{proof} Let $\r'$ be any closed subalgebra of $G$ which is subordinate to $f$ such that $\r\subseteq\r'$.
Note first that $\a\subseteq \z(\n)\subseteq \r\subseteq\r'$, so that for $X\in \r'$ we get
$f([X,Y])=0$ for all $Y\in \a$. Since $[X,Y]\subseteq \z$ for all $Y\in \a\subseteq \z_2(\G)$ and since $f$ is faithful on $\z$, we see that $[X,Y]=0$ for all $Y\in \a$. It follows that $\r'\subseteq \n$, and hence 
that $\r'=\r$ since $\r$ is a maximal $f|_\n$-subordinate subalgebra of $\n$. 

For the proof of (1) let $g=\Ad^*(x)(f)-f\in \G^*$ for some $x\in G$ and let $X=\log(x)$.
 Then we get
\begin{equation}\label{eq-perp}
g(Y)=f\big(\Ad(x^{-1})(Y)- Y\big)=f\big(-[X,Y]+\sum_{n=2}^k\frac{1}{n!}\ad(-X)^n(Y)\big).
\end{equation}
This certainly vanishes whenever $X,Y\in \r$, so that $g\in \r^\perp$ for all $x\in R$.
Conversely, assume that $x\in G$ such that $g=\Ad^*(x)(f)-f\in \r^{\perp}$.
Then (\ref{eq-perp}) implies that 
 $f([X,Y])=0$ for all $Y\in \a$ and hence $[X,Y]=0$ for all $Y\in \a$, since $[X,\a]\subseteq\z$ and 
$f$ is faithful on $\z$. Thus $X\in \n$ and $x\in N$.  
Since $g|_\n=\Ad^*(x)f|_\n- f|_\n\in \r^\perp|_\n$, it follows from the assumption that $x\in R$.

For the proof of (2) we first remark that if $(g_n)_{n\in \NN}$ is a sequence in 
$\G^*$ and $g\in \G^*$ such that
$g_n|_\n\to g|_\n$ in $\n^*$, then, after passing to a subsequence if necessary, we 
find $h_n\in \n^{\perp}$ such that $g_n+h_n\to g$ in $\G^*$. 
To see this we identify 
$\G^*$ with $\widehat{\G}$, $\h^*$ with $\widehat{\h}$ and $\h^{\perp}$ with the annihilator of $\h$ in 
$\G^*$ via $f\mapsto\epsilon\circ f$. The claim then follows from the well-known isomorphism $\widehat{\G}/\h^{\perp}\cong \widehat{\h}$ and the  openness of the quotient map $\widehat{\G}\to\widehat{\G}/\h^{\perp}$.

Assume now that $f\in \G^*$ and $h\in \r^\perp$. By the assumption we may approximate 
$(f+h)|_\n$ by a sequence $\Ad^*(x_n)f|_\n$ for some sequence $(x_n)_{n\in \NN}$ in $R$.
By the above remark we may pass to a subsequence, if necessary, to find elements $h_n\in \n^{\perp}$ 
such that $\Ad^*(x_n)f+h_n\to f+h$ in $\G^*$.  So the result will follow, if we can show that 
$\Ad^*(x_n)f+\n^\perp\subseteq \overline{\Ad^*(R)f}$ for all $n\in \NN$. But since 
$\Ad^*(x)$ acts as the identity on $\n^\perp$ for all $x\in R$, we may apply $\Ad^*(x_n^{-1})$ to this equation in order to see that
it suffices to show that $f+\n^{\perp}\subseteq \overline{\Ad^*(R)f}$. Indeed, we are going to show that
$f+\n^{\perp}=\overline{\Ad^*(A)f}$ with $A=\exp(\a)$. Since $A\subseteq R$, this gives the result.

To show that $f+\n^{\perp}=\overline{\Ad^*(A)f}$ we first observe that for $X\in \a$ and $Y\in \G$ we have
$$\big(\Ad^*(\exp(X))f\big)(Y)=f\big(\exp(\ad(-X))(Y)\big)=f(Y+[Y,X])$$
and hence we get $\Ad^*(\exp(X))f=f+\Psi_f(\cdot, \dot{X})$, where 
$\Psi_f: \G/\n\times \a/\z\to \m$ is the bicharacter  of  (\ref{eq-nondegenerate1}).
The result follows then from Remark \ref{remark-nondegenerate}.

Finally, for the proof of  (3) let $f'=f+h$ for some $h\in \r^\perp$.  We already saw above that  $\r$ is a polarizing subalgebra for $f'$ if it is one for $f'|_\r$. If $(f+\r^{\perp})|_\n=\mathcal O^R(f|_\n)$, then
 $\mathcal O^R(f'|_\n)=\mathcal O^R(f|_\n)=(f+\r^{\perp})|_\n=(f'+\r^{\perp})|_\n$. 
 Since $\mathcal O^R(g|_\n)\subseteq \overline{\Ad^*(R)g|_\n}$ for all $g\in \G$, it follows then from (2)
 that $f'\in f+\r^\perp\subseteq \overline{\Ad^*(R)f}$ and $f\in f'+\r^\perp\subseteq \overline{\Ad^*(R)f'}$,
 thus $f'\in \mathcal O^R(f)$.

\end{proof}

\begin{remark}\label{rem-quotient}
In general, if $f\in \G^*$ is arbitrary, we cannot expect $f$ to be faithful on the center $\z(\G)$ of $\G$.
But by Lemma \ref{largest-ideal-in-kernel-of-character} we can find an ideal $\j$ in the kernel of $f$
such that $f$ factors through an element $\tilde{f}\in (\G/\j)^*$ which is faithful 
on $\z(\G/\j)$. Let $J=\exp(\j)$ and let $A$, $N$ and $R$ in $G$ such that $A/J$,  $N/J$ and $R/J$ satisfy 
the assumptions of the above lemma. Then it is an easy exercise to show that all conclusions of the lemma also hold for the groups $R$, $N$ and $G$, since all statements  ``factor'' through $R/J$, $N/J$, and $G/J$.
\end{remark}

Using Lemma \ref{lem-polarization} together with the above remark, we can now give an explicit 
construction of a special kind of polarizations as follows: 

\begin{remark}\label{rem-standard}
Let $(G,\G)$ be an $(\m,\epsilon)$-dualizing $k$-Lie pair and let 
$f\in \G^*$. Then we can give a recursive construction of a polarizing subalgebra for $f$
as follows:  We start by putting $(G_0, \G_0):=(G,\G)$ and $f_0:=f$. Then, if $(G_i,\G_i)$ and 
$f_i\in \G_i^*$ are constructed for some $i\in \NN_0$, we construct $(G_{i+1},\G_{i+1})$ and $f_{i+1}\in \G_{i+1}^*$ 
by the following steps:
\begin{enumerate}
\item Choose an ideal $\j_i$ in $\G_i$ which lies in the kernel of $f_i$ such that $f_i$ factors through a functional  $\tilde f_i\in (\G_i/\j_i)^*$ which is faithful on $\z(\G_i/\j_i)$ (this is possible by Lemma \ref{largest-ideal-in-kernel-of-character}). Then pass to $(\tilde{G}_i,\tilde{\G}_i):=(G_i/J_i,\G_i/\j_i)$ with $J_i=\exp(\j_i)$.
\item Choose a maximal abelian subalgebra $\a_i$ of $\z_2(\tilde\G_i)$ and let $\dot \G_{i+1}$ be 
its centralizer in $\tilde\G_i$; 
\item  Put $\G_{i+1}:=q_i^{-1}(\dot \G_i)$, where $q_i:\G_i\to \tilde\G_i$ is the quotient map, put $G_{i+1}:=\exp(\G_{i+1})$ and  $f_{i+1}:=f_i|_{\G_{i+1}}$.
\end{enumerate}
Then after each recursion step, the quotient group $G_i/J_i$ reduces its nilpotence length  at least by one. Let $m$ be the smallest integer such that $G_m/J_m$ is abelian, in which case we see 
that $\G_m$ is a polarizing subalgebra for $f_m$. Then, in view of Lemma \ref{lem-polarization} and   Remark \ref{rem-quotient} we arrive at a sequence
of closed exponentiable subgroups 
$$G=G_0\supseteq G_1\supseteq G_2\supseteq \cdots \supseteq G_m=R$$
and a corresponding sequence of subalgebras $\G_i=\log(G_i)$ with 
the following properties
\begin{enumerate}
\item $G_{i+1}$ is normal in $G_i$ and $G_i/G_{i+1}$ is abelian for all $0\leq i\leq m-1$;
\item $\r$ is a polarizing subalgebra for $f_i:=f|_{\G_i}$  for all $0\leq i\leq m$;
\item $R=\{x\in G_i: \Ad^*(x)f_i\in f_i+\r^{\perp}|_{\G_i}\}$ and
 $f_i+\r^\perp|_{\G_i}=\mathcal O^R(f_i)$  for each $0\leq i\leq m$.
\end{enumerate}
In particular,
  $\r$ is a polarizing subalgebra for $f\in \G^*$,  $R=\{x\in G: \Ad^*(x)f\in f_i+\r^{\perp}\}$ and 
  and $f+\r^\perp$ coincides with the $\Ad^*(R)$-quasi-orbit $\mathcal O^R_f$ in $\G^*$.
\end{remark}

\begin{definition}\label{def-standard}
A polarizing subalgebra  $\r$ for $f$ which is constructed by the above recursion procedure is called a 
{\em standard polarization of grade $m=:m(f,\r)$}.  Note that we get 
$m(f_i,\r)=m-i$ if $f_i\in \G_i^*$ 
is the $i$-th functional  in the above described recursion!
\end{definition}

The standard polarizing subalgebras  as defined above are well adjusted to the Kirillov map:

\begin{proposition}\label{prop-standard-polarization}
Let $(G,\G)$ be an $(\m,\epsilon)$-dualizable $k$-Lie pair and let $f\in \G^*$. Suppose that $\r\subseteq \G$ is a standard polarizing subalgebra for $f$, let $R=\exp(\r)$ and let $\varphi_f\in \widehat{R}$ be the character 
corresponding to $f$. Then
$\ind_R^G\varphi_f$ is irreducible with central character $\varphi_f|_Z$, with $Z=Z(G)$.

Conversely, if $P\in \Prim(G)$ is any primitive ideal in $C^*(G)$, there exists some $f\in \G^*$ and some standard polarization $\r$ of $f$ such that $P=\ker(\ind_R^G\varphi_f)$.
\end{proposition}
\begin{proof} 
We perform induction on the degree $m=m(f, \r)$. If $m=0$, we have $\r=\G$ and nothing is to prove.
If $m>0$, then let $(G_1,\G_1)$ be as in the recursion procedure for the construction of $\r$ and let 
$f_1=f|_{\G_1}$. Then $\r$ is a standard polarizing subalgebra for $f_1$ with $m(f_1,\r)=m-1$, 
and it follows from induction that 
$\rho:=\ind_R^{G_1}\varphi_f$ is irreducible.

To see that $\ind_{G_1}^G\rho$ is irreducible as well let $\j=\j_0\subseteq \ker f$ denote the ideal in step (1) of 
Remark \ref{rem-standard} and let $J=\exp(\j)$.
Then passing to $G/J$ and $G_1/J$ if necessary, we may assume without loss of generality that 
$f$ is faithful on $\z=\z(\G)$ and $G_1=N$ is the centralizer of some maximal abelian subgroup $A$ of 
$\z_2(\G)$. Then Remark \ref{remark-nondegenerate} implies that the bicharacter
$\Phi_f:G/N\times A/Z\to\TT; \Phi_f(\dot x,\dot y)=\varphi_f\big((x;y)\big)$ is nondegenerate. By induction, the central character of $\rho=\ind_R^{G_1}\varphi_f$ 
equals $\varphi_f|_{Z(G_1)}$, and it follows then from Lemma \ref{lemma2} that 
$\ind_R^G\varphi_f=\ind_{G_1}^G\rho$ is irreducible with central character $\varphi_f|_Z$.

We prove the second assertion by induction on the nilpotence length $l$ of $G$. If $l=1$, then $G$ is abelian, and nothing is to prove. So assume $l>1$ and 
that $P\in \Prim(G)$ is given.  Choose $\pi\in \widehat{G}$ with $P=\ker\pi$.
Let $J\subseteq G$ be a maximal exponentiable normal subgroup of $G$ which lies in the kernel of $\pi$.
If $J$ is nontrivial, we pass to $(G/J,\G/\j)$ with $\j=\log(J)$. So assume $J$ is trivial. Let $\psi\in \widehat{Z}$ 
be the central character of $\pi$ and let $g\in \z^*$ such that $\psi=\varphi_g$ on $Z$. 
It follows then from Lemma \ref{lemma-faithful-central-character} that  $g$ is faithful on $\z^*$.
Then Remark \ref{remark-nondegenerate} implies that the bicharacter $\Phi_\psi:G/N\times A/Z\to\TT$ 
is nondegenerate and  Proposition \ref{A-N-proposition} implies that there exists 
some $\rho\in \widehat{N}$ with $\rho|_Z\sim \psi$, $\ind_N^G\rho$ is irreducible, and 
$\ker\pi=\ker(\ind_N^G\rho)$. 
Since $N$ has nilpotence length smaller than $l$, we can assume by induction that there exists a 
functional $f_1\in \n^*$ and a standard polarization $\r$ for $f_1$ such that $\ker\rho=\ind_R^N\varphi_{f_1}$.
Choose any $f\in \G^*$ which restricts to $f_1$ on $N$. It follows then from the first part of this 
proposition that  $\psi\sim \rho|_Z\sim \varphi_{f_1}|_Z=\varphi_f|_Z=\varphi_g$, and hence that $f|_\z=g$ is faithful on $\z$. But it follow then from our constructions and Lemma \ref{lem-polarization} that $\r$ is a standard polarization for $f$ and that $\ind_R^G\varphi_f=\ind_N^G(\ind_R^N\varphi_{f_1})$ is irreducible with $\ker(\ind_R^G\varphi_f)=\ker\pi=P$.
\end{proof}

The above proposition suggests, that for every $(\m,\epsilon)$-dualizable $k$-Lie pair $(G,\G)$
 there is a  surjective map 
\begin{equation}\label{Kirillov-map}
\kappa:\G^*\to\Prim(G);\; f\mapsto\ker(\ind_R^G\varphi_f)
\end{equation}
where $R=\exp(\r)$ for some standard polarizing subalgebra $\r$ for $f$. But in order to have this map well defined, we need to show that the $C^*$-kernel of the induced representation $\ind_R^G\varphi_f$ does not depend on the particular choice of the standard polarizing algebra for $f$. 
We first do the case  of two-step nilpotent groups:

\begin{lemma}\label{lem-two-step}
Suppose that $(G,\G)$ is an $(\m,\epsilon)$-dualizable $k$-Lie pair such that $G$ is two-step nilpotent and let 
$f\in \G^*$ such that $f$ does not vanish on the center $Z=Z(G)$. Then, if  $\mathfrak r$ is a polarizing subalgebra for $f$ in $\G$, then $R=\exp(\mathfrak r)$ is a polarizing subgroup for $\psi:=\varphi_f|_Z$ in $G$ (see the discussion preceeding Proposition \ref{prop-two-step}).

Moreover, if $\r$ and $\r'$ are any two polarizing subalgebras for $f$ with corresponding subgroups 
 $R=\exp(\r)$ and $R'=\exp(\r')$ of $G$, then 
 $\ind_R^G\varphi_f\sim\ind_{R'}^G\varphi_f$.
\end{lemma}
\begin{proof} It is clear that if $\mathfrak r$ is a polarizing subalgebra for $f$, then $R=\exp(\mathfrak r)$ 
is subordinate to $\psi$. To see that $R$ is a polarizing subgroup for $\psi$, it suffices to show that 
the homomorphism $\Phi_\psi:G/R\to\widehat{R}$ given by $\Phi_\psi(\dot x)(y)=\psi\big((x;y)\big)$ is injective.
Suppose it's not. Let $x=\exp(X)\in \G\setminus \mathfrak r$ such that $\Phi_\psi(\dot{x})\equiv 1$.
By the Campbell-Hausdorff formula this implies that $\epsilon\circ f([X, Y])=1$ for all $Y\in \mathfrak r$, hence
 $f([X,Y])=0$ for all $Y\in \mathfrak r$ by Lemma \ref{lemma-subalgebras-give-lie-pairs} .
But this contradicts the fact that $\mathfrak r$ is a maximal subordinate algebra for $f$. Te second assertion follows now from 
Corollary \ref{cor-polarizing}.
\end{proof}

The above lemma is used in the proof of

\begin{proposition}\label{prop-polarization-independence}
Suppose that $(G,\G)$ is an $(\m,\epsilon)$-dualizable $k$-Lie pair. Let $f\in \G^*$, let $\r$ be a standard polarization for $f$ and let $\s\subseteq \G$ be any closed 
subalgebra of $\G$ which is subordinate to $f$. Then $\ind_R^G\varphi_f\prec\ind_S^G\varphi_f$
where $R=\exp(\r)$ and  $S=\exp(\s)$.

In particular, if  $\r$ and $\s$ are both standard polarizations for $f$, then $\ind_R^G\varphi_f\sim\ind_S^G\varphi_f$ and the Kirillov map (\ref{Kirillov-map}) is well defined.
\end{proposition}
\begin{proof}
We give a proof by induction on the nipotence length $l$ of $G$. If $l=1$ the result is trivial, and the case $l=2$ follows from the above lemma. So assume now that $l\geq 3$.
By an application of Zorn's lemma, we can first choose some polarization $\s'$ for $f$ which contains $\s$.
Let $S'=\exp(\s')$. Since $S'$ is amenable, it follows from Remark \ref{rem-group-weak-containment} that 
$\varphi_f\prec\ind_S^{S'}\varphi_f$ as representations of $S'$, and since induction preserves weak containment, we see that $\ind_{S'}^G\varphi_f\prec \ind_{S'}^G(\ind_S^{S'}\varphi_f)$.
Thus if we can check that $\ind_R^G\varphi_f\prec 
\ind_{S'}^G\varphi_f$, the result will follow. So from now on we may assume as well that $\s$ is a polarizing subalgebra for $f$. 

Let $\j$ denote the largest ideal of $\G$ which lies in $\ker f$. 
Since $\s+\j$ is an $f$-subordinate subalgebra of $\G$, we have $\j\subseteq \s$ by maximality of $\s$ and by the same reason we have $\j\subseteq\r$. Thus, we may pass to $(G/J, \G/\j)$, $J=\exp(\j)$, to assume that $f$ is faithful on $\z=\z(\G)$. 
Since $\r$ is a standard polarization, we find a maximal abelian subalgebra $\a$ of $\z_2(\G)$ such that 
$\a$ is contained in $\r$, and then $\r$ is also a standard polarization for $f_1= f|_{\n}$ in $\n$, where 
$\n$ denote the centralizer of $\a$ in $\G$.  Then $N=\exp(\n)$ has nilpotence length smaller than $l$ and
if $\s\subseteq \n$,  it follows from induction that $\ind_R^N\varphi_f\prec\ind_S^N\varphi_f$, which then 
implies $\ind_R^G\varphi_f\prec\ind_S^G\varphi_f$ by induction in steps.

Sa assume now that $\s$ does not lie completely in $\n$. Since $\a$ is an ideal in $\G$, we see that
$\h:=\overline{\s+\a}$ is a closed subalgebra of $\G$. Let 
$\u:=\h\cap \n$ ($=\overline{(\s\cap \n)+\a}$).
Then $\u$ is subordinate to $f$ since clearly $\s\cap \n$ and $\a$ are subordinate to $f$ and since 
$[\s\cap \n,\a]\subseteq[\n,\a]=\{0\}$. Since $\u\subseteq \n$, we have $\ind_R^G\varphi_f\prec \ind_U^G\varphi_f$, for $U=\exp(\u)$ by induction.

In what follows next, we want to argue that $\ind_U^G\varphi_f\sim \ind_S^G\varphi_f$. 
If $H:=\exp(\h)$, then this follow from inducing representation in steps if we can show that
$\ind_U^H\varphi_f\sim \ind_S^H\varphi_f$. We want to do this by showing that after passing to a suitable 
quotient group $H/K$, the problem reduces to the two-step nilpotent case.

For this let $\k:=\ker f|_{\s\cap\n}$.
Then $\k$ is an ideal in $\h$, since $\n\cap \s$ is clearly an ideal of $\h=\overline{\s+\a}$, and 
for all $X\in \s\cap\n$, $Y\in \s$ and $Z\in \a$ we have $f([X, Y+Z])=f([X,Y])+f([X,Z])=0$
since $\s$ is subordinate to $f$ and $\n$ centralizes $\a$. It thus follows that 
$[\k,\h]=[\k,\overline{\s+\a}]\subseteq \k$.
Let $K=\exp(\k)$. Then $K\subseteq U, S$ and $\varphi_f$ vanishes on $K$.
By passing to $(H/K,\h/\k)$ if necessary, we may 
 therefore assume without loss of generality that $f$ is faithful on $\s\cap \n$. 
 We already checked above that $f([\s\cap \n, \h])=\{0\}$, which by faithfulness of $f$ on $\s\cap \n$ implies that 
$[\s\cap \n, \h]=\{0\}$, so $\s\cap \n$ lies in the center of $\h$. 
On the other hand, we have $[\h,\h]\subseteq\s\cap \n$. To see this it suffices to show 
that $[\s+\a,\s+\a]= [\s,\s]+[\s,\a]$ lies in $\s\cap \n$. But $[\s,\s]\subseteq \s\cap\n$ since $\G/\n$ is abelian
and $[\s,\a]\subseteq \z(\G)\subseteq\s\cap \n$, since $\a\subseteq \z_2(\G)$.
It is clear that $\s$ is a  polarizing subalgebra for $f$ in $\h$.
So the result will follow from Lemma \ref{lem-two-step} if we can show that 
$\u$ is also a polarizing subalgebra for $f$ in $\h$. 
Suppose that this  is not the case. Then there exists $X\in \h\setminus \n$ 
such that $f([X,\u])=\{0\}$, which in particular implies that $f([X,\a])=\{0\}$ and then $[X,\a]=\{0\}$, since 
$[X,\a]\subseteq\z(\G)$ and $f$ is faithful on $\z(\G)$. But this implies $X\in \n$, a contradiction.
\end{proof}

\begin{remark}\label{rem-polarization-independent}
If $\s\subseteq \G$ is an arbitrary polarization for $f\in \G^*$ and $\r$ is a standard polarization, then 
the above theorem does not imply that $\ind_R^G\varphi_f\sim\ind_S^G\varphi_f$, it only gives the 
weak containment $\ind_R^G\varphi_f\prec\ind_S^G\varphi_f$. If we could show that the 
algebra $\u=\overline{(\s\cap \n)+\a}$ is actually a polarizing subalgebra for $f$ in $\G$ (and not just in $\h$),
the stronger weak equivalence result would follow immediately from the same kind of induction argument.
\end{remark}


%

\section{(Bi)-Continuity of the Kirillov
map}\label{section-the-kirillov-homeomorphism}

In this section we want to show that the Kirillov map 
$\kappa:\G^*\to \Prim(G)$ of  (\ref{Kirillov-map})
is continuous and factors through a map 
\begin{equation}\label{Kirillow-quasi}
\kappa: \G^*/\sim\to  \Prim(G)
\end{equation}
where $\G^*/\sim$ denotes the quasi-orbit space for the coadjoint action of $G$ on $\G^*$. 
Recall that two elements $f,f'\in \G^*$ are in the same $\Ad^*(G)$-quasi orbit, if either is in the closure
of the $\Ad^*(G)$-orbit of the other.

In order to prove continuity, we need to recall Fell's topology on the 
subgroup-representation space $\mathcal S(G)=\{(H,\rho): \text{$H$ a closed subgroup of $G$ and $\rho\in \Rep(H)$}\}$ of a locally compact group $G$. 

In the following, let $G$ be a locally compact group and let
$\mathcal{K}(G)$ denote the set of all closed subgroups of $G$
equipped with the compact-open topology as introduced in \cite{Fell}.
For later use, let us recall the following description of convergence  in $\mathcal K(G)$:

\begin{lemma}[{\cite{Fell}}]\label{lem-convergence}
Assume that the  net $(K_i)_{i\in I}$ 
converges to $K\in \mathcal K(G)$ and let $x\in G$.
Then $x\in K$ if and only if there 
exists a subnet $(K_{j})_{j}$ of $(K_i)_{i\in I}$ and elements $x_j\in K_j$
for all $j\in J$ such that $x_j\to x$ in $G$. 
\end{lemma}

If $G$ is second countable, the same is true for $\mathcal K(G)$, and we can replace nets by sequences 
in the above result.
A smooth choice of Haar measures in
$\mathcal{K}(G)$ is a mapping $K \mapsto \mu_K$ assigning to each
$K$ in $\mathcal{K}(G)$ a left Haar measure $\mu_K$ on $K$
such that $K\mapsto \int_K f(x) d\mu_Kx$ is continuous for all $f\in C_c(G)$ -- it is
shown in  \cite{Fel} that they always exist.
Let $Y$ be the set of all pairs
$(K,x)$, where $K \in \mathcal{K}(G)$ and $x \in K$. Then $Y$ is a closed subset of $\mathcal{K}(G) \times G$, hence locally compact in the relative topology. Let
$\{\mu_K \}$ be a fixed smooth choice of Haar measures on
$\mathcal{K}(G)$. 
%

Let $\Delta_K$ be the modular function for the closed subgroup
$K$ of $G$. Then $(K,x) \mapsto \Delta_K(x)$ is a continuous
function on $Y$. We make $C_c(Y)$ into a normed $*$-algebra with
the following definitions of convolution, involution and norm given by
\begin{equation*}
\begin{array}{rcl}
(f*g)(K,x) & = & \int_K f(K,y)\; g(K,y^{-1}x) \;d \mu_K (y),  \medskip \\
f^*(K,x)   & = & \overline{f(K,x^{-1})} \Delta_K (x^{-1}), \; \hbox{and} \medskip \\
\| f \|    & = & \sup_{K \in \mathcal{K}(G)} \int_K |f(K,x)| \; d
\mu_K (x).
\end{array}
\end{equation*}
Each element of $C_c(Y)$ can be thought of as a function on
$\mathcal{K}(G)$, whose value at $K$ is in the group algebra of
$K$. The operations are pointwise. The completion $A_s(G)=\overline{C_c(Y)}$ in the 
above defined norm is a Banach $*$-algebra and its enveloping $C^*$-algebra $C_s^*(G)$  is called 
the {\em subgroup algebra} of $G$. For each $K\in \mathcal K(G)$ we obtain a 
canonical $*$-homomorphism $\Phi_K: C_s^*(G)\to C^*(K)$ given on the level of $C_c(Y)$
by $f\mapsto f(K,\cdot)\in C_c(K)$. Then 
each unitary representation $\pi$ of $K$ 
can be lifted to the  $*$-representation $\pi\circ \Phi_K$ of $C^*_s(G)$.
 In this way we may
regard  $\mathcal S(G)$ 
as a subset of $\Rep(C_s^*(G))$. 

The {\em Fell  topology} on $\mathcal{S}(G)$  is the 
restriction of the Fell topology on $\Rep(C_s^*(G))$ to $\mathcal S(G)$ via this identification.

\begin{remark}\label{rem-SG-properties}
In the following we list some important properties of the Fell topology on $\mathcal S(G)$:
\\
{\bf (a)} 
The topology of $\mathcal{S}(G)$ is independent of the particular
smooth choice of Haar measures $\{ \mu_K \}$ (\cite[ Lemma 2.3]{Fel}).
\\
{\bf (b)}
The projection $\mathcal S(G)\to \mathcal K(G); (K, \pi) \mapsto K$  is continuous ({\cite[Lemma 2.5]{Fel}}).
\\
{\bf (c)} 
For each $K$ in $\mathcal{K}(G)$, the mapping $\pi \mapsto (K,
\pi)$ is a homeomorphism from $\Rep(K)$ onto its image in  $\mathcal{S}(G)$ ({\cite[Lemma 2.6]{Fel}}).
\\
{\bf (d)}
Every irreducible $*$-representation of $C_s^*(G)$ is of the form $\pi\circ \Phi_K$
for some unique $(K, \pi)$ in $\mathcal{S}(G)$, $\pi
\in \hat{K}$ ({\cite[Lemma 2.8]{Fel}}).
\end{remark}

In \cite[\S 3]{Fel}, the Fell topology on $\mathcal{S}(G)$ is
described in terms of functions of positive type on subgroups. As
a consequence, one obtains the continuity of restricting  and inducing representations from and to
varying subgroups.

\begin{theorem}[{\cite[Theorem 3.2]{Fel} and \cite[Theorem 4.2]{Fel}}]\label{continuity-of-induction-restriction}
Let
\[\mathcal{W}:=\{(H,K,\pi) \mid (K ,\pi) \in \mathcal{S}(G),\; H
\in \mathcal{K}(G),\; H \subseteq K\}\subseteq \mathcal{K}(G) \times \mathcal{S}(G).
\]
Then the map
$(H,K, \pi) \mapsto (H, \pi|_H)$ from $\mathcal{W}$ to
$\mathcal{S}(G)$ is continuous.
 Similarly, if
\[
\mathcal{W}:= \{(H,K, \pi) \mid (K, \pi) \in \mathcal{S}(G),\; H
\in \mathcal{K}(G),\; H \supseteq K\}\subseteq \mathcal{K}(G) \times \mathcal{S}(G)
\]
Then
$(H,K, \pi) \mapsto (H, \ind_K^H \pi)$ from $\mathcal{W}$ to
$\mathcal{S}(G)$ is continuous.
\end{theorem}

The following lemma is a direct  consequence of \cite[Theorem 3.1']{Fel} (and the remark following that theorem):

\begin{lemma}\label{lemma-convergent-characters}
Let $(H_n, \chi_n)$ be a sequence in $\mathcal{S}(G)$, let $(H,
\chi) \in \mathcal{S}(G)$, and suppose that $\chi, \chi_n$, $n \in
\N$, are characters. Then the following are equivalent:
\begin{itemize}
\item[(i)] $(H_n, \chi_n) \rightarrow (H, \chi)$ in
$\mathcal{S}(G)$.
\item[(ii)] $H_n \rightarrow H$ in $\mathcal{K}(G)$ and for every
subsequence $(H_{n_k})$ of $(H_n)$ and every element $h_{n_k} \in
H_{n_k}$ with $h_{n_k} \rightarrow h$ for some $h \in H$, one has
$\chi_{n_k}(h_{n_k}) \rightarrow \chi(h)$ in $\mathbb{C}$.
\end{itemize}
\end{lemma}

\begin{remark}\label{rem-expo-limits}
Suppose that $(G,\G)$ is a $k$-Lie pair and that $(H_n)_{n\in \NN}$ is a sequence of closed exponentiable subgroups of $G$. Let $\h_n=\log(H_n)$ denote the corresponding closed subalgebras of $\G$. 
Assume that $H_n\to H$ for some closed subgroup of $G$.
Regarding 
$\G$ as a locally compact abelian group, it follows from Lemma \ref{lem-convergence} and the fact that $\exp:\G\to G$ is a homeomorphism, that $\h_n\to \h=\log(H)$ in $\mathcal K(\G)$, and another easy application of 
Lemma \ref{lem-convergence} shows that $\h$ is a subalgebra of $\G$. In particular, it follows that the exponentiable subgroups of $G$ are closed in $\mathcal K(G)$.
\end{remark}

Using all these preparations, we are now able to prove

\begin{proposition}\label{continuity-of-kappa}
Let $(G,\G)$ be a $(\m,\epsilon)$-dualizable nilpotent $k$-Lie pair. Then the
Kirillov map
$$
\kappa : \G^* \longrightarrow \Prim(C^*(G)),\; f \mapsto
\ker(\ind_R^G \varphi_{f}),
$$ is continuous.
\end{proposition}

\begin{proof}
Let $(f_n)_{n \in \N}$ be a sequence in $\G^*$ and suppose that
$f_n \rightarrow f$ in $\G^*$ for some $f \in \G^*$. For every $n
\in \N$, let $\mathfrak{r}_n$ be a standard  polarizing subalgebra for
$f_n$ and define $R_n:= \exp(\mathfrak{r}_n)$. We may regard $(\mathfrak{r}_n)_{n\in \NN}$
as a sequence in the compact space $\mathcal K(\G)$ of closed subgroup of $\G$,
and we may hence assume, after passing to a subsequence if necessary, 
 that $\r_n\to \s$ for some subalgebra $\s$ of $\G$ and $R_n\to S=\exp(\s)$.

 We claim that $\mathfrak{s}$ is
$f$-subordinate. For this, let $X$ and $Y$ be two arbitrary
elements of $\mathfrak{s}$. By Lemma
\ref{lem-convergence} and by passing to a suitable
subsequence if necessary,
 we can find for every $n \in \N$, elements $X_n,Y_n
\in \mathfrak{r}_n$, such that $X_n\to X$ and $Y_n\to Y$ in $\G$. Since the
commutator is continuous, we see that  $[X_n,Y_n] \to  [X,Y]$
 and therefore
\[
0= f_n ([X_n,Y_n]) \rightarrow f([X,Y]).
\]
It follows then from Proposition \ref{prop-polarization-independence} that if $\r$ is any standard polarization 
for $f$ and $R=\exp(\r)$, then $\ind_R^G\varphi_f\prec \ind_S^G\varphi_f$. 
Moreover, if $x_n\in R_n$ such that $x_n\to x$ for some $x\in S$, then $X_n:=\log(x_n)\to \log(x)=:X$ 
and hence $f_n(X_n)\to f(X)$ in $\m$. But this implies that
$$\varphi_{f_n}(x_n)=\epsilon\circ f_n(X_n)\to\epsilon\circ f(X)=\varphi_f(x),$$
which by Lemma \ref{lemma-convergent-characters} proves that $(R_n,\varphi_{f_n})\to (S,\varphi_f)$
in $\mathcal S(G)$. By Theorem \ref{continuity-of-induction-restriction} we see that 
$\ind_{R_n}^G\varphi_{f_n}\to \ind_S^G\varphi_f$ in $\Rep(G)$, and since 
$\ind_R^G\varphi_f\prec \ind_S^G\varphi_f$ it follows  from
Remark \ref{rem-weak-containment} that $\ind_{R_n}^G\varphi_{f_n}\to \ind_R^G\varphi_f$
in $\widehat{G}$. But then we also get 
$$\kappa(f_n)=\ker(\ind_{R_n}^G\varphi_{f_n})\to \ker(\ind_R^G\varphi_f)=\kappa(f)$$
in $\Prim(C^*(G))$. Thus $\kappa$ is continuous.
\end{proof}

Suppose that $(G,\G)$ is an $(\m,\epsilon)$-dualizable nilpotent $k$-Lie pair. 
 Let $\r\subseteq \G$ be  a standard polarizing subalgebra for a given $f\in \G^*$ and let $x\in G$. 
Then  one easily checks that $\Ad(x)(\r)$ is a standard polarizing subalgebra for $\Ad^*(x)(f)$ and that 
$$\varphi_{\Ad^*(x)f}(y)=\varphi_f(x^{-1}yx)=x\cdot\varphi_f(y)$$
 for all $y\in \exp(\Ad(x)\r)=xRx^{-1}$, where $R=\exp(\r)$. 
Thus it follows from Remark \ref{rem-properties-of-induction} that 
$$\ind_R^G\varphi_f\cong
\ind_{xRx^{-1}}^G\varphi_{\Ad^*(x)f},$$
 which implies that the 
Kirrolov map is constant on $\Ad^*(G)$-orbits in $\G^*$. 

Suppose now that $f$ and $f'$ are in the same $\Ad^*(G)$-quasi-orbit in $\G^*$. Recall that this means that 
$f'\in \overline{\Ad^*(G)f}$ and $f\in \overline{\Ad^*(G)f'}$. 
Since $\kappa$ is constant on orbits, it follows from the continuity of $\kappa$ that
$\kappa(f)\in \overline{\kappa(f')}$ and vice versa. Since $\Prim(C^*(G))$ is a T$_0$-space, this implies 
that $\kappa(f)=\kappa(f')$. Hence  we get

\begin{corollary}\label{tilde-kappa-is-well-defined}
Let $(G, \G)$ be an $(\m,\epsilon)$-dualizable nilpotent $k$-Lie pair. Then the
{\em Kirillov-orbit map}
\[
\tilde{\kappa}: \G^*/\!_{\sim} \rightarrow \Prim(C^*(G)), \;
\mathcal{O} \mapsto \ker(\ind_{R}^G \varphi_{f}),
\]
where $f \in \G^*$ is any chosen representative of the coadjoint
quasi-orbit $\mathcal{O}$, is a well-defined continuous and surjective map.
\end{corollary}


In the rest of this section we want to show that the Kirilov-orbit map of Corollary \ref{tilde-kappa-is-well-defined} is a homeomorphism, at least if $G$ satisfies the following regularity condition:

\begin{definition}\label{def-regular}
A $k$-Lie pair $(G,\G)$ is called {\em regular} if for any two closed subalgebras $\h$ and $\r$ of  $\G$ such that
 $[\h,\r]\subseteq \h$, the 
 sum $\h+\r$ is closed in $\G$. 
 \end{definition}

Note that if $\h$ and $\r$ are as in the definition, then $\h+\r$ is a closed subalgebra of $\G$ 
 and if $H=\exp(\h)$ and $R=\exp(\r)$, then 
 $HR=\exp(\h+\r)$ is a closed subgroup of $G$. 

For the proof of openness of the Kirillov-orbit map under this extra condition we rely heavily on the ideas of  Joy \cite{Joy}, in which convergence of a sequence 
in $\widehat{G}$ is described in terms of convergence of corresponding subgroup representations in 
$\mathcal S(G)$. 
Regularity of $(G,\G)$ is used in the proof of  the following lemma



\begin{lemma}\label{inducing-pair-lemma}
Let $(G,\G)$ be an $(\m,\epsilon)$-dualizable $k$-Lie pair and 
let $H$ be a closed normal, exponentiable subgroup of $G$ such 
 that $G/H$ is abelian. Let $f\in \G^*$, 
 $\pi\in \widehat{G}$ and  $\rho\in \widehat{H}$ with   $\ker(\pi)= \kappa(f)$ and 
$\ker(\rho)=\kappa(f|_\h)$, where $\h=\log(H)$.
 Then $\pi\prec \ind_H^G\rho$ and if $(G,\G)$ is regular, we also have $\rho\prec \pi|_H$.
 \end{lemma}
  \begin{proof}
  Let $\r$ be a standard polarizing subalgebra of $\G$ for $f$ and let 
 $\s$ be a standard polarizing subalgebra of $\h$ for  $f|_\h$. Let $R=\exp(\r)$ and $S=\exp(\s)$.
 Since $\s$ is subordinate to $f$ it follows from Proposition \ref{prop-polarization-independence} that 
 $\pi\sim\ind_R^G\varphi_f\prec\ind_S^G\varphi_f\cong \ind_H^G(\ind_S^H\varphi_f)\sim \ind_H^G\rho$,
 which proves the first assertion.
 
 Assume now that $(G,\G)$ is regular. Then $HR$ is a closed normal subgroup of $G$ and it follows from 
 Remark \ref{rem-group-weak-containment} and induction in steps that
 $\ind_R^{HR}\varphi_f\prec (\ind_R^G\varphi_f)|_{HR}\sim \pi|_{HR}$. 
 This implies that $(\ind_R^{HR}\varphi_f)|_H\prec\pi|_H$ and it suffices to show that 
 $\rho\prec (\ind_R^{HR}\varphi_f)|_H$.
 By Lemma \ref{lem-ind-res} we have $(\ind_R^{HR}\varphi_f)|_H\cong \ind_{R\cap H}^H\varphi_f$ and 
 since $\r\cap \h$ is clearly subordinate to $f|_\h$, we see from Proposition \ref{prop-polarization-independence} that $\rho\sim \ind_S^H\varphi_f\prec\ind_{R\cap H}^H\varphi_f\cong  (\ind_R^{HR}\varphi_f)|_H$, which finishes the proof.
 \end{proof}

%
%


The following lemma gives the main step in the proof of the openness of the Kirillov map.

\begin{lemma}\label{lem-open}
Suppose that $(G,\G)$ is a regular $(\m,\epsilon)$-dualizable $k$-Lie pair. Let  $(f_n)_{n\in \NN}$ be a sequence in $\G^*$ and 
let $f\in \G^*$ such that $\kappa(f_n)\to\kappa(f)$ in $\Prim(C^*(G))$.
Let $\r_n$ be a standard polarizing subalgebra for $f_n$ and let $R_n=\exp(\r_n)$.
Then, after passing to a subsequence if necessary,  there exists a closed subgroup $S$ of $G$ such that 
$\s=\log(S)$ is subordinate to $f$, and a sequence $(x_n)_{n\in \NN}$ in $G$ such that
$$(x_nR_nx_n^{-1}, \varphi_{\Ad^*(x_n)f_n})\to (S, \varphi_f)\quad\text{in $\mathcal S(G)$.}$$
\end{lemma}
\begin{proof}
After passing to a subsequence, if necessary, we may assume that all standard polarizations $\r_n$ have the same degree $m=m(f_n,\r_n)$, as defined in Definition \ref{def-standard}. It follows then from 
Remark \ref{rem-standard} that, for each $n\in \NN$, we find a sequence of subgroups
$$G=G_n^0\supseteq G_n^1\supseteq G_n^2\supseteq \cdots \supseteq G_n^m=R_n$$
with the properties as listed in that remark. By definition, we have $\kappa(f_n)=\ker(\ind_{R_n}^G\varphi_{f_n})$ for all $n\in \NN$.

Now fix some $i<m$ and assume that there exists a closed subgroup $H^i$ of $G$ such 
that $(G_n^i, \ind_{R_n}^{G_n^i}\varphi_{f_n})\to (H^i, \ind_{S^i}^{H^i}\varphi_f)$ in $\mathcal S(G)$, 
where $S^i=\log{\s^i}$ for some standard polarizing subalgebra $\s^i$ of $\h^i=\log(H^i)$ for $f|_{\h^i}$.
After passing to a subsequence, if necessary, we may assume that $G_n^{i+1}\to H^{i+1}$ for some 
closed subgroup $H^{i+1}$ of $G$. 
We claim that $H^{i+1}$ is normal in $H^i$ and that $H^i/H^{i+1}$ is abelian. For this let $x,y\in H^i$.
After passing to a subsequence we may assume that there are $x_n,y_n\in G_n^i$ such that $x_n\to x$ and 
$y_n\to y$. Then $x_ny_nx_n^{-1}y_n^{-1}\to xyx^{-1}y^{-1}$ and it follows from Lemma \ref{lem-convergence} 
that $xyx^{-1}y^{-1}\in H^{i+1}$. This proves the claim. 

Let $\h^{i+1}=\log(H^{i+1})$ and let
$\s^{i+1}$ be a standard polarizing subalgebra 
for $f|_{\h^{i+1}}$. 
By continuity of restriction, we see that
$$(G_n^{i+1}, (\ind_{R_n}^{G_n^i}\varphi_{f_n})|_{G_n^{i+1}})\to 
(H^{i+1}, (\ind_{S^i}^{H^i}\varphi_f)|_{H^{i+1}})$$
in $\mathcal S(G)$. By Lemma \ref{inducing-pair-lemma} we have $\ind_{S^{i+1}}^{H^{i+1}}\varphi_f\prec (\ind_{S^i}^{H^i}\varphi_f)|_{H^{i+1}}$. Thus it follows from Remark \ref{rem-weak-containment} 
 that 
$$(G_n^{i+1}, (\ind_{R_n}^{G_n^i}\varphi_{f_n})|_{G_n^{i+1}})\to 
(H^{i+1}, \ind_{S^{i+1}}^{H^{i+1}}\varphi_f)$$
in $\mathcal S(G)$. By Remark \ref{rem-group-weak-containment} we have
$$(\ind_{R_n}^{G_n^i}\varphi_f)|_{G_n^{i+1}}\sim \{x\cdot (\ind_{R_n}^{G_n^{i+1}}\varphi_f): x\in G_n^1\}$$
for all $n\in \NN$, so it follows from part (c) of Remark \ref{rem-weak-containment} that, after passing to a subsequence if necessary, we can find $x_n^i\in G_n^i$ such that 
$$(G_n^{i+1}, x_n^i\cdot(\ind_{R_n}^{G_n^{i+1}}\varphi_{f_n}))\to 
(H^{i+1}, \ind_{S^{i+1}}^{H^{i+1}}\varphi_f).$$
Since 
$$x_n^i\cdot(\ind_{R_n}^{G_n^{i+1}}\varphi_{f_n})\cong \ind_{x_n^iR_n(x_n^i)^{-1}}^{G_n^{i+1}}\varphi_{\Ad^*(x_n)f_n}$$
for all $n\in \NN$, we now see that
$$(G_n^{i+1}, \ind_{x_n^iR_n(x_n^i)^{-1}}^{G_n^{i+1}}\varphi_{\Ad^*(x_n^i)f_n})\to 
(H^{i+1}, \ind_{S^{i+1}}^{H^{i+1}}\varphi_f)$$
in $\mathcal S(G)$.
Now, starting this procedure at $i=0$, where we have the convergent sequence
$$(G, \ind_{R_n}^G\varphi_{f_n})\to (G, \ind_R^G\varphi_f)$$
by the  assumption that $\kappa(f_n)\to \kappa(f)$ in $\Prim(C^*(G))$, 
and passing from $f_n$ to $\Ad^*(x_n^i)f_n$ for suitable 
$x_n^i\in G_n^i$ (and a suitable subsequence of $(f_n)_{n\in \NN}$) 
in each step $i\to i+1$, we will arrive after $m$ steps at a convergent (sub-)sequence
$$(x_nR_nx_n^{-1}, \varphi_{\Ad^*(x_n)f_n})\to (S, \varphi_f)$$
with $x_n=x_n^{m-1}x_n^{m-2}\cdots x_n^0$. Note that we do have the character $\varphi_f$ on the right hand side, since the one-dimensional representations in $\mathcal S(G)$ are closed in $\mathcal S(G)$.
Hence in the step $(m-1)\to m$ of the above procedure we must have $\ind_{S^m}^{H^m}\varphi_f$ a character, which implies that $H^m=S^m=:S$.

\end{proof}

\begin{remark}\label{rem-extension}
Suppose that $G$ is an abelian locally compact group and that $(H_n)_{n\in \NN}$ is a sequence 
of closed subgroups such that $H_n\to H$ in $\mathcal K(G)$. Let $(\chi_n)_{n\in \NN}$ be a sequence 
in $\widehat{G}$ and let $\chi\in \widehat{G}$ such that 
$$(H_n,\chi|_{H_n})\to (H, \chi|_H)\quad\text{in $\mathcal S(G)$.}$$
Then, after passing to a subsequence if necessary, we can find elements $\mu_n\in H_n^\perp$ for all $n\in \NN$ such that $\chi_n\cdot\mu_n\to \chi$ in $\widehat{G}$.

To see this well-known fact we simply use continuity of induction, to see that 
$\ind_{H_n}^G(\chi_n|_{H_n})\to \ind_H^G(\chi|_H)$ in $\Rep(G)$. Since $\chi\prec \ind_H^G(\chi|_H)$
and since $\ind_{H_n}^G(\chi_n|_{H_n})\sim \chi_n\cdot \widehat{G/H_n}=\chi_n\cdot H_n^{\perp}$ for all $n\in \NN$, the result follows from parts (a) and (c) of  
 Remark \ref{rem-weak-containment}.
 \end{remark}

\begin{lemma}\label{lem-extension}
Suppose that $(f_n)_{n\in \NN}$ is a sequence in $\G^*$ and that $(\mathfrak r_n)_{n\in \NN}$ is sequence of closed subalgebras of $\G$ such that each $\r_n$ is subordinate to $f_n$. Let $f\in \G^*$ and let $\s$ be 
a closed subalgebra of $\G$ which is subordinate to $f$. Suppose further that 
$$(R_n,\varphi_{f_n})\to (S,\varphi_f)\quad \text{in $\mathcal S(G)$.}$$
Then, after passing to a subsequence, there exist elemens $g_n\in \r_n^{\perp}$ such that 
$f_n+g_n\to f$ in $\G^*$.
\end{lemma}
\begin{proof}
By passing from $G$ to $\G$ via $\log:G\to\G$, we get from our assumption that 
$$(\r_n, \epsilon\circ f_n|_{\r_n})\to (\s, \epsilon\circ f|_{\s})$$
in $\mathcal S(\G)$, regarding $\G$ as an abelian locally compact group. By the remark, and using the 
isomorphism  $\G^*\cong \widehat{\G}; g\mapsto \epsilon\circ g$, we can find,  after passing to a subsequence if necessary, elements $g_n\in \r_n^{\perp}$ such that 
$$\epsilon\circ (f_n+g_n)=(\epsilon \circ f_n)\cdot (\epsilon \circ g_n)\to \epsilon\circ f$$
in $\widehat{\G}$. But then we also have $f_n+g_n\to f$ in $\G^*$.
\end{proof}

We  can now prove:

\begin{proposition}\label{prop-open}
Suppose that $(G,\G)$ is a regular $(\m,\epsilon)$-dualizable $k$-Lie-pair and assume that 
$(f_n)_{n\in \NN}$ is a sequence in $\G^*$ such that $\kappa(f_n)\to \kappa(f)$ in $\Prim(C^*(G))$ for some $f\in \G^*$. Then, after passing to a subsequence, there exists a sequence
$x_n$ in $G$ such that $\Ad^*(x_n)f_n\to f$ in $\G^*$. 
\end{proposition}
\begin{proof}
It follows from Lemma \ref{lem-open} that, after passing to a subsequence,  and after  passing  from $f_n$ to 
$\Ad^*(y_n)f_n$ for some suitable $y_n\in G$, we may assume that there is a choice of standard regularizations $\r_n$ for $f_n$ and a subalgebra $\s$ subordinate to $f$ such that 
$$(R_n,\varphi_{f_n})\to (S,\varphi_f)$$
in $\mathcal S(G)$. By the above lemma, we can find, after passing to another subsequence, elements  $g_n\in \r_n^\perp$ such that $f_n+g_n\to f$ in $\G^*$. By Proposition \ref{prop-standard-polarization} we know that $\Ad^*(R_n)f_n$ is dense in $f_n+\r_n^{\perp}$
for all $n\in \NN$. Thus, we may approximate $f_n+g_n$ by  $\Ad^*(x_n)f_n$ for a suitable $x_n\in R_n$ to 
obtain $\Ad^*(x_n)f_n\to f$ in $\G^*$.
\end{proof}

As a consequence we now get the main result of this paper

\begin{theorem}\label{thm-open}
Suppose that $(G,\G)$ is a regular $(\m,\epsilon)$-dualizable $k$-Lie-pair. Then the Kirillov-orbit map
$\tilde{\kappa}:\G^*/\!\sim\to \Prim(C^*(G))$ is a homeomorphism.
\end{theorem}
\begin{proof}
Corollary \ref{tilde-kappa-is-well-defined} shows that the Kirillov-orbit map is continuous and surjective and 
the above proposition directly implies that it is open. So the result follows if we can check that it is injective. 
For this suppose that $f,f'\in \G^*$ such that $\kappa(f)=\kappa(f')$. Then the above proposition, applied to 
the constant sequence $f_n=f$ implies that there exists a sequence $(x_n)_{n\in \NN}$ in $G$ such that
$\Ad^*(x_n)f_n\to f'$. Similarly, we can also find a sequence $(y_n)_{n\in \NN}$ in $G$ such that 
$\Ad^*(y_n)f'\to f$. Thus $f$ and $f'$ lie in the same quasi-orbit in $\G^*$.
\end{proof}

\section{GCR and CCR representations}\label{sec-GCR}

Recall that an irreducible representation $\pi\in \widehat{A}$ is called a {\em GCR-representation}
(resp. {\em CCR-representation}), if $\pi(A)$ contains (resp. is equal to) the compact operators 
$\mathcal K(H_\pi)$. Note that this implies that every irreducibe representation $\rho\in \widehat{A}$ with $\rho\sim \pi$ 
must already be unitarily equivalent to $\pi$.

We say that $A$ is GCR (resp. CCR) if every irreducibe representation of $A$ is GCR (resp. CCR).
If $A$ is separable, it follows from Glimm's famous theorem (see \cite[Chapter 12]{Dix}) that $A$ is GCR if and only if $A$ is of type I and 
a representation $\pi\in \widehat{A}$ is GCR (resp. CCR) if and only if $\{\pi\}$ is locally closed (resp. closed) in $\widehat{A}$. (Recall that  a subset $Y$ of a topological space $X$ is called {\em locally closed} if 
$Y$ is open in its closure $\overline{Y}$.)

A locally compact group $G$ is called GCR (resp CCR, resp type I) if the group $C^*$-algebra $C^*(G)$ is 
GCR (resp. CCR, resp type I), and similarly for representations.
In what follows, we prove the following theorem: 

\begin{theorem}\label{theorem-GCR}
Suppose that $(G,\G)$ is a regular $(\m,\epsilon)$-dualizable nilpotent $k$-Lie pair. 
Let $f\in \G^*$,  let $\r\subseteq \G$ be a standard polarizing subalgebra for $f$ and let $R=\exp(\r)$.
Then the following are true:
\begin{enumerate}
\item If $f\in \G^*$ such $\Ad^*(G)f$ is locally closed (resp. closed) in $\G^*$, then $\ind_R^G\varphi_f$ is GCR (resp. CCR).
\item If, in addition, $\Ad^*(R)f=f+\r^{\perp}$, then $\ind_R^G\varphi_f$ is GCR (resp. CCR) if and only if 
$\Ad^*(G)f$ is locally closed (resp. closed).
\end{enumerate}
\end{theorem}

The condition that $\Ad^*(R)f=f+\r^{\perp}$ is certainly necessary for $\Ad^*(G)f$ being locally closed, as we shall see in Lemma \ref{lem-R-orbit} below. However, we do not know whether this condition is automatically true if $\ind_R^G\varphi_f$ is GCR. 

\begin{remark}\label{rem-locally-closed}
Since for a given GCR-representation $\pi$ of $C^*(G)$ any other irreducible  representation $\tau$ with the same kernel must already be equivalent to $\pi$, we see that if $\ind_R^G\varphi_f$ is 
GCR for some $f\in \G^*$ and some standard polarizing subalgebra $\r=\log(R)$, then {\em every representation} $\ind_{R'}^G\varphi_{f'}$ for any $f'\in \mathcal O(f)$ and any standard polarizing algebra 
$R'$ for $f'$ must be equivalent to $\ind_R^G\varphi_f$. In particular, if $(G,\G)$ is regular and 
$C^*(G)$ is GCR, then it follows from 
Theorem \ref{thm-open} that 
$$\widehat{\kappa}:\G^*/\!\sim\to \widehat{G}; \mathcal O(f)\mapsto \ind_R^G\varphi_f$$
is a well-defined homeomorphism. 

This observation can be specialized to locally closed subsets of $\widehat{G}$: If $E\subseteq \widehat{G}$ 
is locally closed, then $E$ is homeomorphic to $E':=\{\ker\pi: \pi\in E\}\subseteq \Prim(C^*(G))$ via the canonical map $\pi\mapsto \ker\pi$. Thus if $\G^*_E\subseteq \G^*$ denotes the inverse image of $E'$ 
under the Kirillov map, we obain a homeomorphism
$$\widehat{\kappa}_E:\G^*_E/\!\sim\to E; \mathcal O(f)\mapsto \ind_R^G\varphi_f.$$
\end{remark}

Since a locally closed orbit $\Ad^*(G)f$ coincides with the quasi-orbit 
$\mathcal O(f)$ of $f$, it follows from Theorem \ref{theorem-GCR} and the above remark that

\begin{corollary}\label{cor-GCR}
Suppose that $(G,\G)$  is a regular $(\m,\epsilon)$-dualizable nilpotent $k$-Lie pair such that all $\Ad^*(G)$-orbits are locally closed (resp. closed) in $\G^*$. Then 
$C^*(G)$ is GCR (resp. CCR) and the Kirrilov-orbit map
$$\widehat{\kappa}: \G^*/\Ad^*(G)\to \widehat{G}; f\mapsto \ind_R^G\varphi_f$$
is a homeomorphism.
\end{corollary}

For the proof of Theorem \ref{theorem-GCR} we need the following two lemmas:

\begin{lemma}\label{lem-R-orbit}
Suppose that $(G,\G)$ is an  $(\m,\epsilon)$-dualizable nilpotent $k$-Lie pair and let $f\in \G^*$ such that 
$\Ad^*(G)f$ is locally closed in $\G^*$. Then $\Ad^*(R)f=f+\r^\perp$.
\end{lemma}
\begin{proof}
Let $G_f=\{x\in G: \Ad^*(x)f=f\}$ denote the stabilizer of $f$. Since $\Ad^*(G)f$ is locally closed, hence locally compact Hausdorff, and since everything in sight is second countable, it follows that $G/G_f$ is homeomorphic to  $\Ad^*(G)f$ via $xG_f\mapsto \Ad^*(x)f$ (e.g. use \cite[Proposition 7.1]{Rief}). Since 
$\Ad^*(x)f=f\in f+\r^{\perp}$ for all $x\in G_f$, it follows that $x\in R$ (see Remark \ref{rem-standard}). Hence, $G_f\subseteq R$ and $R/G_f$ is closed in $G/G_f$, which implies 
$\Ad^*(R)f$ is closed in $\Ad^*(G)f$, hence locally closed in $\G^*$.
 It thus follows that the the $\Ad^*(R)$-quasi-orbit $\mathcal O^R(f)$ coincides with the orbit
$\Ad^*(R)f$ and it follows then from Remark \ref{rem-standard} that $\Ad^*(R)f=f+\r^{\perp}$.
\end{proof}

\begin{lemma}\label{lem-GCR}
Let $N$ be a closed normal subgroup of the second countable locally compact group $G$ such that $G/N$ is abelian. Let $\rho\in \widehat{N}$ such the stabilizer $G_P$  of $P=\ker\rho$ for the action of $G$ on $\Prim(N)$ is equal to $N$. Then the following are equivalent:
\begin{enumerate}
\item $\pi$ is GCR (resp. CCR).
\item $\rho$ is GCR and the orbit $G\cdot\rho=\{x\cdot \rho: x\in G\}$ is locally closed (resp. closed) in 
$\widehat{N}$.
\end{enumerate}
\end{lemma}
\begin{proof}
By Green's theory (e.g. see \cite[Chapter 1]{Ech}) there exists a twisted action $(\alpha,\tau)$ of $(G,N)$ on 
$C^*(N)$ such that $C^*(G)\cong C^*(N)\rtimes_{\alpha,\tau}G$ and such that induction and restriction for
the twisted crossed product $C^*(N)\rtimes_{\alpha,\tau}G$ is compatible with induction and restriction 
of representations in the group $G$  between any subgroups of $G$ which contain $N$. 
The lemma then translates into \cite[Lemma 3.2.2]{Ech}.
\end{proof}

\begin{proof}[Proof of Theorem \ref{theorem-GCR}]
Let $f\in \G^*$ and let $\r$ be a standard polarization of $f$ of degree $m=m(f,\r)$. We give the proof 
by induction on the degree $m$.  If $m=0$ we have $R=G$ and assertions (1) and (2) of the theorem are
trivially true.

Assume now $m>0$. Let $\j\subseteq \G$ be the largest ideal in the kernel of $f$ and let $J=\exp{\j}$.
Then $f$ factors through a functional $\tilde{f}\in (\G/\j)^*$ and one one easily checks that all assertions are true for $(G,\G)$ and $f$ if and only if they are true for 
$(G/J,\G/\j)$ and $\tilde{f}$. Thus, by Lemma \ref{largest-ideal-in-kernel-of-character} we may  assume without loss of generality that $f$ is faithful on $\z=\z(\G)$.

Let $\a$ be a maximal abelian subalgebra of $\z^2(\G)$, let $\n$ be its centralizer in $\G$ and let 
$A=\exp(\a)$ and $N=\exp(\n)$.
It follows then from Lemma \ref{lem-polarization} and Remark \ref{rem-standard} that $\r$ is
a standard polarization for $f|_\n$ of degree $m-1$. Hence we may assume by induction that 
Theorem \ref{theorem-GCR} holds for $(N,\n)$ and $f|_\n$.

Let $\rho\in \ind_R^N\varphi_f$. We claim that the stabilizer for the action of $G$ on $P=\ker\rho$ is 
equal to $N$. Indeed, if $\psi:=\varphi_f|_A$, then it is shown in Proposition \ref{prop-standard-polarization} 
that $\rho|_A=\psi\cdot \id_{H_\rho}$ and we know 
from Remark \ref{remark-nondegenerate} that the bicharacter 
$\Phi_\psi: G/N\times A/Z\to \TT; (\dot{x},\dot {y})\mapsto \psi(xyx^{-1}y^{-1})$ is nondegenerate. 
In particular, the map $G/N\to \widehat{A/Z}$ which sends $\dot{x}$ to the character 
$\chi_x=\Phi_\psi(\dot x,\cdot)$ is injective.
Assume now that $x\in G$ such that $\ker(x\cdot \rho)=\ker\rho$. Then 
$x\cdot\psi\sim x\cdot\rho|_A\sim \rho|_A\sim \psi$, hence $x\cdot \psi=\psi$, which implies 
that $\chi_x(y)=\psi(y^{-1}x^{-1}yx)=\big(\overline{\psi}\cdot (x\cdot\psi)\big)(y)=1$ for all $y\in A$, from which it follows 
that $x\in N$.

Thus we may apply Lemma \ref{lem-GCR} to see that $\ind_R^G\varphi_f=\ind_N^G\rho$ is GCR (resp. CCR)
if and only if $\rho=\ind_R^N\varphi_f$ is GCR and the orbit $\{x\cdot\rho: x\in G\}$ is locally closed (resp. closed) in $\widehat{N}$.

Assume now that the orbit $\Ad^*(G)f$ is locally closed (resp. closed)  in $\G^*$. 
Let $G_f=\{x\in G: \Ad^*(x)f=f\}$ denote the stabilizer of $f$ in $G$. As observed in the proof of Lemma \ref{lem-R-orbit} we have 
$G/G_f\cong \Ad^*(G)f$ via $xG_f\mapsto \Ad^*(x)f$ and since $G_f\subseteq R\subseteq N$
it follows that $\Ad^*(N)f$ is also locally  closed in $\G^*$. Since $\r\subseteq \n$ it follows from Lemma \ref{lem-R-orbit} that
 $$f+\n^\perp\subseteq f+\r^{\perp}=\Ad^*(R)f\subseteq \Ad^*(N)f\subseteq \Ad^*(G)f.$$
  The same argument works if we replace $f$ by $\Ad^*(x)f$ and $R$ by $xRx^{-1}\subseteq N$ 
    for any $x\in G$, from which it follows that $\Ad^*(G)f=\Ad^*(G)f+\n^\perp$ and 
 $\Ad^*(N)f=\Ad^*(N)f+\n^\perp$. 
 Since $\n^*$ carries the quotient topology with respect to the projection
 $\res: \G\to \n; f\mapsto f|_\n$ (which becomes clear after identifying $\G^*$ with $\widehat{\G}$ and $\n^*$ with $\widehat{\n}$), we see that $\Ad^*(G)f|_\n$ and $\Ad^*(N)f|_\n$   are locally closed in $\n^*$
 (closed if $\Ad^*(G)f$ is closed in $\G^*$).
Since the theorem holds for $(N,\n)$,  this implies that $\rho:=\ind_R^N\varphi_f$ is GCR, and since $\kappa^{-1}(\{\ker(x\cdot\rho): x\in G\}=\Ad^*(G)f|_\n$ is locally closed in $\n^*$, 
 it follows from Theorem \ref{thm-open} combined with Remark \ref{rem-locally-closed} that the orbit 
 $G\cdot\rho$ is locally closed   in $\widehat{N}$ (closed if $\Ad^*(G)f$ is closed).  
 Thus, $\ind_R^G\varphi_f=\ind_N^G\rho$ is GCR by 
 Lemma \ref{lem-GCR} (and CCR if  $\Ad^*(G)f$ is closed).
 
 Assume now for the converse that $\pi=\ind_R^G\varphi_f$ is GCR and that $\Ad^*(R)f=f+\r^\perp$. 
 This property is certainly invariant under conjugation with elements in $x\in G$, so we have
 $$\Ad^*(xRx^{-1})\Ad^*(x)f=\Ad^*(x)f+\Ad(x)(\r)^\perp$$ for all $x\in G$. It follows as above that 
  $\Ad^*(G)f=\Ad^*(G)f+\n^\perp$.
  
Since  $\ind_R^G\varphi_f$ is GCR, it follows from Lemma \ref{lem-GCR}
 that $\rho=\ind_R^N\varphi_f$ is GCR and the orbit $G\cdot\rho$ is locally closed in $\widehat{N}$ (closed if 
 $\pi$ is CCR). Applying Theorem \ref{theorem-GCR} to $(N,\n)$ 
 this implies  that $\Ad^*(N)f|_\n$ is locally closed in $\n^*$.
 Since for all $x\in G$ the representation $x\cdot\rho\cong\ind_{xRx^{-1}}^N\varphi_{\Ad^*(x)f|_\n}$ is also GCR we see by induction  that  the $\Ad^*(N)$-orbits of $\Ad^*(x)f|_\n$ are locally closed for all $x\in G$ and hence coincide with the respective quasi-orbits $\mathcal O^N({\Ad^*(x)f|_\n})$. It follows then from Theorem \ref{thm-open} that 
 $$\Ad^*(G)f|_\n=\bigcup_{x\in G}\Ad^*(N)\Ad^*(x)f|_\n=\bigcup_{x\in G}\mathcal O^N({\Ad^*(x)f|_\n})=
 \kappa^{-1}(\{\ker(x\cdot\rho) : x\in G\})$$  is locally closed in $\n^*$ (closed if $\pi$ is CCR). 
But since  $\Ad^*(G)f=\Ad^*(G)f+\n^\perp$, this implies that $\Ad^*(G)f$ is locally closed in $\G^*$ 
(closed if $\pi$ is CCR).

\end{proof}

\section{Examples}\label{section-examples}

In this section we want to discuss some examples for our generalized Kirillov theory.
We start with the case of unipotent groups over $\RR$ or $\QQ_p$:

%

\begin{example}\label{example-Qp}
Let $K$ be a local field of characteristic zero (i.e., $K=\RR,\CC$ or a finite extension of $\Q_p$ for some 
prime $p$). Let $\G$ be any finite dimensional nilpotent Lie algebra over $K$ and let 
$G=\G$ with multiplication given by the Campbell-Hausdorff formula.
Then $\exp:\G\to G$ identifies with the mapping $X\mapsto x$. 
It follows that $(G,\G)$ is an $\infty$-Lie pair (or rather, a $\QQ$-Lie pair)   in the sense of Definition \ref{definition-k-Lie-pair}.
Note that every unipotent nilpotent Lie group over $K$ can be realized in this way and that in case $K=\RR$ this class coincides with the class of connected and simply connected nilpotent real Lie groups.

Since every finite dimensional Lie algebra over a finite extension of $K=\RR$ or $K=\QQ_p$ is also a finite dimensional Lie algebra over $K$, we assume from now on that $K=\RR$ or $K=\Q_p$.
Let $\epsilon:\RR\to \TT$ denote the basic character $\epsilon(t)=e^{2\pi i t}$ for $t\in\RR$ and 
\[
\epsilon: \Q_p \rightarrow \mathbb{T}, \; \sum_{j=m}^{\infty} c_j
\; p^j \mapsto \exp(2 \pi i \sum_{j=m}^{-1} c_j)
\]
for $x=\sum_{j=m}^\infty c_jp^j\in \Q_p$. If $V$ is a finite dimensional  vector space $V$ over $K$, then it follows from  \cite[Theorem 3]{Wei} that
 $V^*:=\Hom(V,K)\cong \widehat{V}$ via $f\mapsto \epsilon\circ f$. In particular, we obtain 
 $\G^*\cong\widehat{\G}$ via $f\mapsto \epsilon\circ f$. Since, $\QQ$ is dense in $K$, we 
 see that every every closed $\QQ$-subalgebra $\h$
 of $\G$ is also a $K$-subalgebra. By basic linear algebra, this  implies that every $f\in \h^*$ extends to a 
 functional $\tilde f\in \G^*$, so we see that $(G,\G)$ is $(K,\epsilon)$-dualizable in the sense of Definition \ref{def-dualizable}. Moreover, since the sum of two $K$-subalgebras $\h$ and $\n$ is a finite dimensional subspace of $\G$, it must be closed in $\G$. This shows that $(G,\G)$ is also regular in the sense of
Definition \ref{def-regular}. It is well known that the $\Ad^*(G)$-orbits in $\G^*$ are always closed, since they 
can be described as the set solutions of certain polynomial equations. 
Thus, as a consequence of Corollary \ref{cor-GCR} we see that $C^*(G)$ is CCR for all such $G$, and the
Kirillov-orbit map
$$\widehat{\kappa}:\G^*/\Ad^*(G)\to \widehat{G}; f\mapsto \ind_R^G\varphi_f$$
is a homeomorphism.
 Thus, we recover the original case of Kirillov \cite{Kir} (together with the main result of 
\cite{Bro}) in case of  real groups. Unipotent groups over $\Q_p$ have been considered by Moore in 
\cite{Moo} and (for the homeomorphism result) by Howe in \cite{How}.
 \end{example}

\begin{example}\label{quasi-p-groups}
Recall from \cite{How} that a locally compact group $G$ is called a 
{\em quasi-$p$ group} if every $x\in G$ generates a compact subgroup of $G$
which is a  projective limit of finite $p$-groups. For instance, every unipotent group over $\Q_p$,
as considered in the previous example, is a quasi-$p$ group in this sense.
Howe has shown in \cite[Theorem I]{How} that every nilpotent quasi-$p$ group 
is totally disconnected and 
has a unique $k$-Lie-algebra $\mathfrak g$
such that the exponential map $\exp:\mathfrak g\to G$ is bijective 
and satisfies the Campbell-Hausdorff formula,  if $G$ (and hence $\mathfrak g$)
is of nilpotence length $k<p$. Thus, in this case 
 $(G,\mathfrak g)$ is a nilpotent $k$-Lie pair. 
 
 Moreover, the underlying abelian group 
$\mathfrak g$ is also a quasi-$p$ group (since $\exp$ sends the closed subgroup generated 
by $X\in \mathfrak g$ to the closed subgroup generated by $x=\exp(X)\in G$). It follows 
that every character of $\mathfrak g$ takes its values in a cyclic group of order a power of $p$.
Thus if we let $\m:=\{\zeta\in \TT: \zeta^{p^m}=1\;\text{for some}\; m\in \NN\}$ (which is known as 
Pr\"ufer's $p$-group) equipped with the discrete topology, and if $\epsilon:\m\to \TT$ denotes the inclusion map,
then one easily checks that $(G,\mathfrak g)$ is $(\m,\epsilon)$-dualizable.  
Finally, since  the sum $\mathfrak a+\mathfrak b$
of two closed subgroups $\mathfrak a$ and $\mathfrak b$ in the 
abelian totally disconnected group $\G$ is always closed (since the sum of the intersection of both groups 
with a fixed compact open subgroup $\mathfrak c$ of $\G$ is compact and open in $\mathfrak a+\mathfrak b$),
we see that the Lie algebra $\mathfrak g$ is also regular. Thus our results apply to those 
groups and we recover most of the content of  \cite[Theorem II]{How}.
\end{example}

\begin{example}\label{example-unipotent-groups-over-F_p((t))} 
In what follows next, we want to consider unipotent groups over a local field $K$ of positive characteristic. 
Note that this means that $K$  is isomorphic to a function field $F_q((t))$, for some power  $q$ of $p$.
We denote by $Tr_0(n,K)$ (resp.  $Tr_1(n,K)$) the  set of upper triangular $n\times n$-matrices over $K$ with $0$'s (resp. $1$'s) on the diagonal. If $n< p:=\Char(K)$, one can check that the 
exponential map 
$$\exp: Tr_0(n,K)\to Tr_1(n,K); X\mapsto \exp(X)=\sum_{l=0}^n\frac{X^l}{l!}$$
is a bijection with inverse map $\log: Tr_1(n,K)\to Tr_0(n,K); \log(x)=\sum_{l=1}^n \frac{(-1)^l}{l+1} (1-x)^l$.
It is clear that $Tr_0(n,K)$ is a $\Lambda_k$-module for all $k<p$, so we see that
 $\big(Tr_1(n,K),Tr_0(n,K)\big)$ becomes a $k$-Lie pair for all 
$n\leq k<p$.  For any closed subgroup  $G$ 
of $Tr_1(n,K)$ let $\mathfrak g=\log(G)\subseteq Tr_0(n,K)$. Then $G$ is a quasi-$p$ group 
with Lie algebra $\mathfrak g$ as considered in the previous example and our results apply to
the pair $(G,\G)$.

Note that  every unipotent linear algebraic group $G$ over $K$ is isomorphic
to an algebraic subgroup of the upper triangular unipotent group
$Tr_1(n,K)$ for some $n \in \N$ (see \cite[Theorem 4.8]{Bor}). 
\end{example}

\begin{example}
Another important example of our approach is given by the class of countable torsion free discrete
divisible nilpotent groups, as considered by Carey, Moran and Pearce in \cite{CMP}. 
As explained in \cite[\S 2]{CMP}, if $G$ is such a group, then there exists a Lie algebra $\G$ 
over $\Q$ together with an exponential map $\exp:\G\to G$ with inverse map $\log: G\to\G$ which satisfy the Campbell-Hausdorff formula, so $(G,\G)$ is an $\infty$-Lie pair in our notation. 
Particular examples are given 
as follows:
Let $\frak g_{\R}$ be any real nilpotent Lie-algebra with base $\{X_1,\ldots, X_n\}$ and with rational structure coefficients with respect to this base.  Let $\G_{\Q}\subseteq \G_{\R}$ denote the 
$\Q$-vector space spanned by $X_1,\ldots, X_n$ and let $G_{\Q}=\exp(\G_{\Q})\subseteq G_{\R}$, where 
$G_{\R}$ denotes the simply connected and connected nilpotent Lie group corresponding to $\G_{\R}$.
Then $G_{\Q}$ is a countable, torsion free, and  divisible group with Lie algebra $\G_{\Q}$.

We need to show  that there exists a $\Q$-module 
$\frak m$ and a basic character $\epsilon :\frak m\to \mathbb{T}$ such that
the pair $(\frak m,\epsilon)$ satisfies the conditions of Defintion \ref{definition-k-Lie-pair}.
For this we recall that the dual group $\widehat{\Q}$ of $\Q$ can be identified 
with the compact group $\frak m:=A_{\Q}/\Q$, where $A_{\Q}$ is the group of adeles and $\Q$ 
is imbedded diagonally into $A_{\Q}$. The identification is done by
$$[a]\mapsto \epsilon_a\quad\text{with $\epsilon_a(\lambda)=\epsilon_{\infty}(\lambda a_{\infty})\prod_{p\in \P}\epsilon_p(\lambda a_p)$}$$
where $(a_{\infty}, a_2, a_3,\ldots)$ is any representative of $[a]$ in $A_{\Q}$ and $\epsilon_\infty$
is the basic character of $\R$ and $\epsilon_p$ is the basic character of $\Q_p$ for every prime $p\in \P$
(see Example \ref{example-Qp}).
We define $\epsilon: \m\to \mathbb{T}$ by 
$$\epsilon([a]):=\epsilon_a(1).$$
We claim that for any $\Q$-vector space $V$ (viewed as a discrete group), we 
obtain an isomorphism $V^*:=\Hom_{\Q}(V,\m)\cong \widehat{V}$ via $f\mapsto \epsilon\circ f$.
To see that this map is injective suppose that $\epsilon\circ f\equiv 1$ for some $f\in V^*$.
Fix $v\in V$. Then, for each $\lambda\in \Q$ we get
$$1=\epsilon(f(\lambda v))=\epsilon(\lambda\cdot f(v))=\epsilon_{f(v)}(\lambda)$$
so that $\epsilon_{f(v)}$ is the trivial character of $\Q$, which implies that $f(v)=0$.
For surjectivity let $\psi\in \widehat{V}$ and let $\{v_i: i\in I\}$ be a base of $V$.
For each $i\in I$ let $\chi_i:\Q\cdot v_i\to \TT$ denote the restriction of $\chi$ to $\Q\cdot v_i\subseteq V$
and let $[a_i]\in \m$ such that $\chi(\lambda v_i)=\psi_{a_i}(\lambda)$. Define 
$f\in \Hom(V,\m)$ by $f(v_i)=[a_i]$. Then, for $v=\sum_{i\in I}\lambda_iv_i\in V$ we get  
\begin{align*}
\epsilon\circ f(v)&=\epsilon(\sum_{i\in I} \lambda_i [a_i])=\prod_{i\in I}\epsilon([\lambda_i a_i])
= \prod_{i\in I}\psi_{[\lambda_i a_i]}(1)= \prod_{i\in I}\psi_{[a_i]}(\lambda_i)=\prod_{i\in I}\chi_i(\lambda_i v_i)\\
&=\chi(\sum_{i\in I} \lambda_iv_i)=\chi(v)
\end{align*}
It follows in particular that $\G^*\cong \Hom_{\Q}(\G,\m)$ for any Lie algebra $\G$ over $\Q$. 
Thus we see that if $\G$ is the Lie algebra of some torsion free divisible group $G$, then 
$(G,\G)$ is an $(\m,\epsilon)$-dualizable $\Q$-Lie pair. Since $G$ is discrete, it is clearly regular.
Thus we see that for all such groups the Kirillov-orbit map
$$\tilde\kappa: \G^*/\!\sim\to \Prim(C^*(G))$$
is a homeomorphism. This covers the main result of \cite{CMP}.
\end{example}

\section{Appendix on the Campbell-Hausdorff formula}\label{appendix}

In this appendix we present some details for the proofs of some results related to the Campbell-Hausdorff formula as used in this paper. In particular we want to present a proof of Theorem \ref{k-complete-subgroups-give-subalgebras}. we start with a general remark on the Campbell-Hausdorff formula:

\begin{remark}
Suppose that $(G,\G)$ is a $k$-Lie-pair.
For $X,Y\in \G$ we write $\ad(X)(Y):=[X,Y]$.
As part of the definition  we require that the Campbell-Hausdorff formula describes the multiplication
inside the group $G$ using the laws of the Lie algebra $\G$:
 If $X,Y \in \G$ then $\exp(X) \exp(Y)= \exp(Z)$, where the element
$Z= \log(\exp(X) \exp(Y))$ is of the form
\begin{equation*}
Z  =  \sum_{n=1}^l Z_n = \sum_{n=1}^l\left( \frac{1}{n}
\sum_{s+t=n}(Z'_{s,t}+ Z''_{s,t})\right),
\end{equation*}
where
\begin{eqnarray}\label{definition-C-H-formula-part1}
Z'_{s,t} & = & \sum_{
      \begin{smallmatrix}
        s_1+ \dots + s_m=s \\
        t_1+ \dots + t_{m-1}=t-1 \\
        s_i + t_i \geq 1 \; \forall i    \\
        s_m \geq 1
      \end{smallmatrix}}
 \frac{(-1)^{m+1}}{m} \frac{\ad(X)^{s_1} \ad(Y)^{t_1} \dots
 \ad(X)^{s_m}(Y)}{s_1 !t_1!\dots s_m!}
\end{eqnarray}
and
\begin{eqnarray}\label{definition-C-H-formula-part2}
Z''_{s,t} & = & \sum_{
      \begin{smallmatrix}
        s_1+ \dots + s_{m-1}=s \\
        t_1+ \dots + t_{m-1}=t \\
        s_i + t_i \geq 1 \; \forall i
      \end{smallmatrix}}
 \frac{(-1)^{m+1}}{m} \frac{\ad(X)^{s_1} \ad(Y)^{t_1} \dots
 \ad(Y)^{t_{m-1}}(X)}{s_1 !t_1!\dots  t_{m-1}!}.
\end{eqnarray}
Explicitly, the values of the first three homogeneous components
of $Z$ are
$$
Z_1= X+Y,\quad 
Z_2 = \frac{1}{2} [X,Y],\quad\text{and}\quad
Z_3 = \frac{1}{12}([X,[X,Y]]+ [Y,[Y,X]]).$$
Hence we have for all $X,Y \in \G$:
\[
(\exp X) (\exp Y)= \exp(X+Y+ \frac{1}{2}[X,Y]+
\frac{1}{12}[X,[X,Y]] + \frac{1}{12}[Y,[Y,X]]+ \cdots).
\]
A complete description of this formula can be found for example
in \cite{Bour}. Since the Lie algebra $\G$ is nilpotent, all
sums which appear above are finite.
\end{remark}

In what follows, a commutator $[X_1, \dots,X_m]$ of length $m \geq 1$ in $\G$ is defined
inductively by
\[
[X_1]:= X_1 \; \hbox{and} \; [X_1, \dots,X_m]:=[[X_1, \dots
,X_{m-1}],X_m], \quad X_i \in \G,\; i=1, \dots,m.
\]
and similarly, we define commutators  $(x_1; \dots;x_m)$ of length
$m \geq 1$ in $G$. The following important result follows from \cite[Theorem 6.1.6]{Jen}.

\begin{proposition}\label{prop-commutators-in-log=log-of-commutators}
Let $(G, \G)$ be a nilpotent $k$-Lie pair. If
$x_1, \dots, x_m \in G$ then
\begin{equation}\label{equation-commutators-in-log=log-of-commutators}
\log \bigl( (x_1; \dots; x_m) \bigr)= [\log(x_1), \dots,\log(x_m)]+
\sum_{j} F_j,
\end{equation}
where each $F_j$ is a $\Lambda_k$-linear
combination of commutators $[\log(x_{i_1}), \dots,
\log(x_{i_j})]$ of length $j>m$ and $i_l \in \{1,\dots,m \}$ for
$1 \leq l \leq j$, such that each of $1, \dots,m$ occurs at least
once among the $i_l$.
\end{proposition}

From this proposition we deduce

\begin{lemma}\label{lemma-log(h)+D=log(h')+sum-G_t}
Let $(G, \G)$ be a nilpotent $k$-Lie pair  and let
$H$ be a $k$-complete subgroup of $G$. Let $h \in H$, let $D$ be
any commutator of length $r \geq 1$ with entries in $\log(H)$,
and let $\mu \in \Lambda_k$. Then there exists an element
$h' \in H$ and there exist finitely many $\Lambda_k$-linear combinations, $G_t$, of
commutators of length $t \geq r+1$ with entries in $\log(H)$,
such that
\[
\log(h)+ \mu D= \log(h')+ \sum_{t} G_t.
\]
\end{lemma}

\begin{proof}
Let $r \geq 1$ be fixed and let $D=[\log(x_1), \dots, \log(x_r)]$
for some $x_{i} \in H$, $i=1, \dots, r$. Since 
the commutator is $\Lambda_k$-bilinear, we obtain
\[
\mu [\log(x_1), \dots, \log(x_r)]= [\mu\log(x_{1}), \dots,
\log(x_{r})]= [\log(x_{1}^{\mu}), \dots,
\log(x_{r})],
\]
with  $x_{1}^{\mu} \in H$ since $H$ is $k$-complete. 
 Recall the Campbell-Hausdorff formula
\[
\log(xy)= \log(x)+ \log(y) + \sum_{t \geq 2} G_t,
\]
where each $G_t$ is a linear combination of commutators of length
$t \geq 2$ in $\log(x)$ and $\log(y)$. Applying this formula to
 $x=h$ and $y= (x_{1}^{\mu}; \dots;x_{r})$ yields
\begin{equation}\label{equation-C-H-for-h-and-G}
\log \bigl( h \cdot (x_{1}^{\mu}; \dots;x_{r}) \bigr)= \log(h)+
\log \bigl( (x_{1}^{\mu}; \dots;x_{r}) \bigr)+ \sum_t  G_t,
\end{equation}
where each $G_t$ is a linear combination of commutators of length
$t \geq 2$ in $\log(h)$ and $\log \bigl( (x_{1}^{\mu},
\dots,x_{r}) \bigr)$.  Proposition
\ref{prop-commutators-in-log=log-of-commutators} gives
\[
\log \bigl( (x_{1}^{\mu}; \dots ; x_{r}) \bigr) =[\log(x_{1}^{\mu}),
\dots,\log(x_{r})]+ \sum_s  F_{s},
\]
where each $F_{s}$ is a linear combination of commutators of
length $s>r$ in $\log(x_{1}^{\mu})$ and in $\log(x_{j})$, $j \in
\{2,\dots,r\}$. 
Thus every term $G_t$ in (\ref{equation-C-H-for-h-and-G}) is
in fact a commutator of length $t \geq r+1$ in $\log(h)$,
$\log(x_{1}^{\mu})$, and $\log(x_{j})$, $j \in \{2, \dots, r\}$
and we obtain
\begin{equation}\label{equation-step-h+D->h'+G_t}
\log(h)+ \log \bigl( (x_{1}^{\mu}; \dots; x_{r}) \bigr)  =  \log
\bigl( h \cdot (x_{1}^{\mu}; \dots;x_{r}) \bigr) + \sum_t G_t,
\end{equation}
where each $G_t$ is a linear combination of commutators of length
$t \geq r+1$ in $\log(h)$, $\log(x_{1}^{\mu})$, and
$\log(x_{j})$, $j \in \{2, \dots, r\}$. Since $h':=h
(x_{1}^{\mu}; \dots ; x_{r}) \in H$ the result follows.
\end{proof}

The next lemma completes the proof of Theorem
\ref{k-complete-subgroups-give-subalgebras}. A similar result can
be found in \cite[Chapter II, Exercises 6]{Bour}.

\begin{lemma}\label{lemma-k-complete-subgroup->subalgebra}
Let $(G, \G)$ be a nilpotent $k$-Lie pair, and let
$H$ be a $k$-complete subgroup of $G$. Then $\log(H)$ is a subalgebra 
of $\G$.  
\end{lemma}
\begin{proof} 
Note first that $\lambda\log(x)=\log(x^\lambda)$ for all $\lambda\in \Lambda_k$ since 
$H$ is $k$-complete. We now show that $\log(H)$ is closed under addition. 
For this let $x,y\in H$. The above lemma implies that
$$\log(x)+\log(y)=\log(x_1)+\sum_t G_t$$
for some $x_1\in H$ and a finite sum $\sum_t G_t$ where $G_t$ is a $\lambda_k$-linear combination 
of commutators of length $t\geq 2$ with entries in $\log(H)$. Applying the same lemma again and again 
to each summand $\mu D$ of $G_t$ for $t=2$, we finally obtain an element $x_2\in H$ such that
$$\log(x)+\log(y)=\log(x_2)+\sum_s F_s$$
where each $G_t$ is a $\lambda_k$-linear combination 
of commutators of length $s\geq 3$ with entries in $\log(H)$.
After a finite number of steps we obtain an element $x_l\in H$ such that
$$\log(x)+\log(y)=\log(x_l)+\sum_r E_r$$
where each $E_r$ is a $\lambda_k$-linear combination of commutators of length $r\geq l+1$.
If $l$ is the nilpotence length of $\G$, then all those commutators $E_r$ vanish, and we 
get $\log(x)+\log(y)=\log(x_l)\in \log(H)$.

If we apply Lemma \ref{lemma-log(h)+D=log(h')+sum-G_t}
 to $h=1$ and $D=[\log(x),\log(y)]$, a similar argument shows that $[\log(x),\log(y)]\in \log(H)$
 for all $x,y\in H$, which finishes the proof.
\end{proof}

A similar proof as given in Lemma \ref{lemma-k-complete-subgroup->subalgebra} gives the
following 

\begin{corollary}\label{corollary-inversion-formulas}
For every nilpotent $k$-Lie pair $(G, \G)$ there
exist two different inversion formulas of the Campbell-Hausdorff
formula. One formula expresses the sum of two elements of
$\log(G)$ as an element of $\log(G)$:
\begin{equation}\label{inversion-C-H-sum}
\log(x) + \log(y)  =  \log \bigl( \prod_{m =1}^{k} C_m(x,y)
\bigr),
\end{equation}
where each $C_m(x,y)$ is a product of commutators $(z_1;
\dots;z_m)$ of length $m$ and where each $z_i$ is equal to some
rational power $\lambda \in \Lambda_k$ of some product in
$x$ and $y$. 

The other formula expresses the commutator of two elements of
$\log(G)$ as an element of $\log(G)$:
\begin{equation}\label{inversion-C-H-commutator}
[\log(x), \log(y)]= \log( \prod_{m = 2}^{k} D_m(x,y)),
\end{equation}
where each $D_m(x,y)$ is a product of commutators $(z_1;
\dots;z_m)$ of length $m$ and where each $z_i$ is equal to some
rational power by $\lambda \in \Lambda_k$ of either $x$ or
$y$. 
\end{corollary}

The proof of the following proposition follows  from the above formulas in the usual way and is omitted.

\begin{proposition}\label{prop-log(normal-complete-subgroups)=ideal}
Let $(G, \G)$ be a nilpotent $k$-Lie pair. Then the assignment 
$\mathfrak{n}\mapsto N:=\exp(\mathfrak{n})$ is a bijection between the set of 
closed ideals $\n$ in $\G$ and exponentiable normal subgroups $N$ of $G$.
Moreover, if $N=\exp(\n)$ is an exponentiable normal subgroup of $G$, then 
$(G/N, \G/\n)$ is a nilpotent $k$-Lie pair. 
%
\end{proposition}

We  are now going to show that certain characteristic subgroups 
of $G$ are exponentiable. We start with another important consequence of 
Proposition \ref{prop-commutators-in-log=log-of-commutators}. Of course, the proof is the same 
as in the of ordinary Lie groups, but for completeness, we give the arguments.

\begin{lemma}\label{lem-commutator}
Suppose that $(G,\G)$ is a nilpotent $k$-Lie pair.  Let $x=\exp(X)$ and $y=\exp(Y)$ 
for some $X,Y\in \G$. Then
$(x;y)=1$ if and only if  $[X,Y]=0$. In particular, we have
$Z(G)=\exp(\z(\G))$, where $Z(G)$ and $\z(\G)$ denote the centers of $G$ and $\G$, respectively,
and therefore $Z(G)$ is an exponentiable subgroup of $G$.
\end{lemma}
\begin{proof} If $[X,Y]=0$, then so are all higher commutators in $X$ and $Y$, and it follows then from
Proposition \ref{prop-commutators-in-log=log-of-commutators} that $(x;y)=\exp(0)=1$.
Conversely, assume that $(x;y)=1$. Then the same is true for
all higher group commutators in $x$ and $y$. 
Let $m$ be the nilpotence length of $\G$. 
Then all commutators of length
$m+1$ in $\G$ vanish. 
Suppose now that $[Z_1,\ldots, Z_m]$ is a commutator of length $m$ with $Z_i\in \{X,Y\}$ 
for all $1\leq i\leq m$. Then Proposition \ref{prop-commutators-in-log=log-of-commutators} implies that
$$[Z_1,\ldots, Z_m]=\log(z_1;\ldots;z_m)+\sum_j F_j$$
with $z_i=\exp(Z_i)$ for all $1\leq i\leq m$ and with each $F_j$ a $\Lambda_k$-linear combination of commutator in $X,Y$ of length $>m$. 
But then all $F_j$ vanish, and since $(z_1;\ldots;z_m)=1$ it follows that $[Z_1,\ldots, Z_m]=0$.
The same argument then shows that commutators of length $m-1$ in $X,Y$ vanish, and by 
induction we then  see that all commutators of length $l\geq 2$ in $X,Y$ vanish. In particular,
it follows that $[X,Y]=0$.
The last assertion follows from the fact that $\z(\G)$ is an ideal in $\G$.
\end{proof}

\begin{lemma}\label{lemma-comparing-the-central-series}
Let $(G, \G)$ be a nilpotent $k$-Lie pair of
nilpotence length $l$. Then
$$
\exp(\mathfrak{z}^{i}(\G)) = Z^{i}(G) \quad \forall \;i=1,\dots,l,
$$
where $\mathfrak{z}^{i}(\G)$ denotes the $i$th element of the
ascending central series of $\G$ and $Z^{i}(G)$ denotes the $i$th
element of the ascending central series of $G$. In particular,
$Z^{i}(G)$ is exponentiable for every $i=1,\dots,l$.
\end{lemma}

\begin{proof} 
This is an easy consequence of Lemma \ref{lem-commutator} and Proposition
 \ref{prop-log(normal-complete-subgroups)=ideal}.
\end{proof}

\begin{lemma}\label{lemma-commutator-subgroup-is-exponentiable-subgroup}
Let  $(G,\G)$ be a nilpotent $k$-Lie pair and let $G':=\overline{(G;G)}$ denote the 
closed commutator subgroup of $G$ and let $\G'=\overline{[\G,\G]}$ the closed commutator subalgebra of $\G$. Then 
$G'=\exp(\G')$.
\end{lemma}
\begin{proof} It follows from Proposition \ref{prop-commutators-in-log=log-of-commutators}
that $(G;G)\subseteq \exp([\G,\G])$  which implies $\overline{(G;G)}\subseteq \exp(\overline{[\G,\G]})$.
The converse follows from equation (\ref{inversion-C-H-commutator}). The result then follows from 
the easy fact that $\overline{[\G,\G]}$ is an ideal in $\G$.
\end{proof}

In view of Proposition \ref{A-N-proposition}, the following lemma is certainly very useful:

\begin{lemma}\label{lemma-log(A)-is-a-subalgebra}
Let $k \geq 2$ and let $(G, \G)$ be a nilpotent $k$-Lie pair of
nilpotence length $l \geq 2$. Then  $A\mapsto \log(A)=:\a$ 
gives a bijective correspondence between the maximal abelian subgroups $A$ of $Z_2(G)$ and the maximal 
abelian subalgebras $\a$ of $\z_2(\G)$. In particular, every maximal abelian subgroup of $Z_2(G)$ is exponentiable.

Moreover, if $N$ is the centralizer of a maximal abelian subgroup $A$ of $Z_2(G)$, then
$N$ is exponentiable and  $\n=\log(N)$ is the centralizer of $\a=\log(A)$ in $\G$.
\end{lemma}
\begin{proof} The first assertion is an easy consequence of Lemma \ref{lem-commutator} together with 
Lemma \ref{lemma-comparing-the-central-series}. The second assertion is a consequence of 
Lemma \ref{lem-commutator}  and the fact that
$\n$ is a subalgebra of $\G$.
\end{proof}

\end{document}